\documentclass{amsart}
\usepackage{graphicx}
\usepackage{color}
\usepackage{mathrsfs}
\usepackage{dsfont}
\usepackage{a4wide}
\usepackage{fullpage}
\def\rr{\mathbb R}

\def\nn{\mathbb N}

\def\HH{\mathcal H}
\def\AA{\mathcal A}

\newcommand {\nc}   {\newcommand}
\nc {\be}   {\begin{equation}} \nc {\ee}   {\end{equation}} \nc
{\beq}  {\begin{eqnarray}} \nc {\eeq}  {\end{eqnarray}} \nc {\beqs}
{\begin{eqnarray*}} \nc {\eeqs} {\end{eqnarray*}}
\def\edc{\end{document}}

\newtheorem{thm}{Theorem}[section]

\newtheorem{lem}[thm]{Lemma}

\newtheorem{prop}[thm]{Proposition}
\newtheorem{rk}[thm]{Remark}

\numberwithin{equation}{section}

\providecommand{\abs}[1]{\lvert#1\rvert}%absolute value
\providecommand{\norm}[1]{\lVert#1\rVert}%norm

\theoremstyle{definition}

\numberwithin{equation}{section}
\begin{document}

\title[Stabilization of non-smooth Bresse systems]
{Polynomial stabilization of non-smooth direct/indirect elastic/viscoelastic damping problem involving Bresse system}

\author{St\'ephane Gerbi}
\address{Laboratoire de Math\'ematiques UMR 5127 CNRS, Universit\'e de Savoie Mont Blanc\\
	Campus scientifique, 73376 Le Bourget du Lac Cedex, France}
\email{stephane.gerbi@univ-smb.fr}

\author{Chiraz Kassem}
\address{Universit\'e Libanaise\\
	Facult\'e des Sciences 1\\
	EDST, Equipe EDP-AN\\
	Hadath, Beyrouth, Liban}
\email{shiraz.kassem@hotmail.com}

\author{Ali Wehbe}
\address{Universit\'e Libanaise\\
	Facult\'e des Sciences 1\\
	EDST, Equipe EDP-AN\\
	Hadath, Beyrouth, Liban}
\email{ali.wehbe@ul.edu.lb}

\date{}

\subjclass[2010]{35B37, 35D05, 93C20, 73K50}
\keywords{Bresse system, Kelvin-Voigt damping,  polynomial stability,  non uniform stability, frequency domain approach.}

\begin{abstract}
	We consider an elastic/viscoelastic problem for the Bresse system with fully Dirichlet or Dirichlet-Neumann-Neumann boundary conditions. 
	The physical model consists of three wave equations coupled in certain pattern.
	The system is damped directly or indirectly by global or local Kelvin-Voigt damping. 
	Actually, the number of the dampings, their nature of distribution (locally or globally) and the smoothness of the damping coefficient at the interface play a crucial
	role in the type of the stabilization of the corresponding semigroup. 
	Indeed, using frequency domain approach combined with multiplier techniques and the construction of a new multiplier function, 
	we establish different types of energy decay rate (see the table of stability results at the end). Our results generalize and improve many 
	earlier ones in the literature (see \cite{ArwadeYoussef}) and in particular some studies done on the Timoshenko system with Kelvin-Voigt 
	damping (see for instance  \cite{GhaderWehbe}, \cite{Tian2017} and \cite{ZhaoLiuZhang}).
\end{abstract}

\maketitle

\tableofcontents	
%%%%%%%%%%%%%
% INTRODUCTION
%%%%%%%%%%%%%
\section{Introduction}\label{se1}
\subsection{The Bresse system with Kelvin-Voigt damping}
Viscoelasticity is the property of materials that exhibit both viscous and elastic characteristics when undergoing deformation. 
There are several mathematical models representing physical damping. The most often encountered type of damping in vibration studies are
linear viscous damping and Kelvin-Voigt damping which are special cases of proportional damping. 
Viscous damping usually models external friction forces such as air resistance acting on the vibrating structures and is thus called ``external damping'', 
while Kelvin-Voigt damping originates from the internal friction of the material of the vibrating structures and thus called ``internal damping".
The stabilization of conservative evolution systems (wave equation, coupled wave equations, Timoshenko system ...) by viscoelastic Kelvin-Voigt type 
damping has attracted the attention of many authors. In particular, it was proved that the stabilization of wave equation with local Kelvin-Voigt damping 
is greatly influenced by the smoothness of the damping coefficient and the region where the damping is localized (near or faraway from the boundary)
even in the one-dimensional case, {see \cite{ChenLiuLiu-1998,Liu-Liu-2002}}. This surprising result initiated the study of an  elastic system
with local Kelvin-Voigt damping %(see Subsection \ref{viscoelastic} for more details). 
There are {a} few number of publications concerning the stabilization of Bresse or Timoshenko systems with viscoelastic Kelvin-Voigt damping.
(see Subsection \ref{Aims} below).

In this paper, we study the stability of Bresse system with localized non-smooth Kelvin-Voigt damping coefficient at the interface and we briefly state
results when the Kelvin-Voigt damping coefficients are either global or localized but smooth at the interface since the tools used for the study
of non-smooth coefficient are used in the same, but  much simpler, way when the coefficients act on the totality of the domain or are smooth enough at the interface.
These results generalize and improve many earlier ones in the literature. 

The Bresse system is usually considered in studying elastic structures of the arcs type (see \cite{JGJ}). It can be expressed by the equations of motion:
$$
\begin{array}{rcl}
	\rho_1\varphi_{tt} & = & Q_x +\ell N \\
	\rho_2 \psi_{tt} & = & M_x-Q \\
	\rho_1w_{tt} & = & N_x-\ell Q 
\end{array}
$$
where
$$
N =  \displaystyle{k_3\left(w_x-\ell \varphi\right)} +F_3,\quad Q =  \displaystyle{k_1\left(\varphi_x+\psi+\ell w\right)} +F_1, \quad M  =  \displaystyle{k_2\psi_{x}}+F_2
$$
$$F_1=D_1\left(\varphi_{xt}+\psi_t+\ell w_t\right), \quad F_2=D_2 \psi_{xt}, \quad F_3=D_3\left(w_{xt}-\ell \varphi_t\right)$$
and where $F_1$, $F_2$ and $F_3$ are the Kelvin-Voigt dampings. When $F_1=F_2=F_3=0$, $N$, $Q$ and $M$ denote the axial force,
the shear force and the bending moment. The functions $\varphi,\ \psi,$ and  $w$ model the vertical, shear angle,
and longitudinal  displacements  of the filament. Here $\rho_1=\rho A, \ \rho_2=\rho I, \ k_1=k'GA, \ k_3=EA, \ k_2=EI, \ \ell =R^{-1}$
where $\rho$ is the density of the material, $E$ is the modulus of elasticity, $G$ is the shear modulus, $k'$ is the shear factor,
$A$ is the cross-sectional area, $I$ is the second moment of area of the cross-section, and $R$ is the radius of curvature.
see figure \ref{fig} reproduces from \cite{ArwadeYoussef}. 
The damping coefficients $D_1$, $D_2$ and $D_3$ are bounded non negative functions over $(0,L)$. 
\begin{center}
	\begin{figure}
\centering{\includegraphics[scale=0.5]{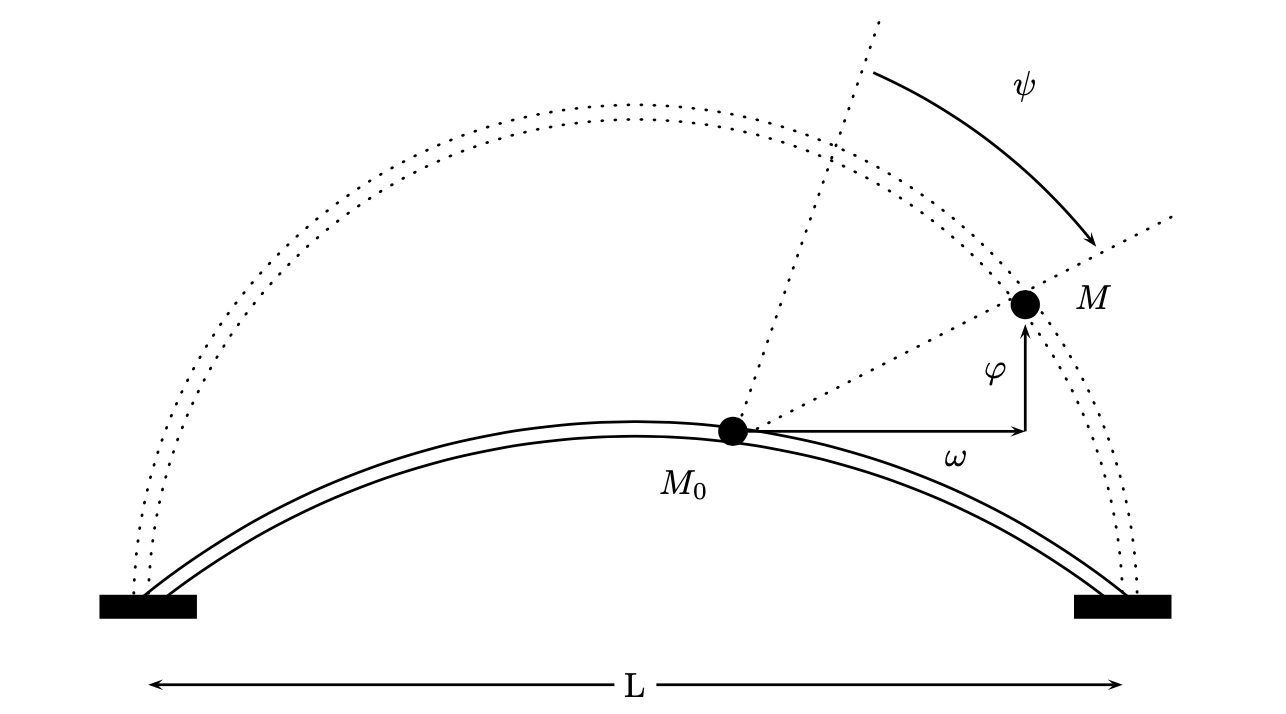}}
\caption{After deformation the particle $M_{0}$ of  the beam is at the position $M$.}
\label{fig}
	\end{figure}
\end{center}
So we will consider the system of  partial differential equations given on $\left(0,L\right)\times\left(0,+\infty\right)$ by the following  form:\\
\begin{equation}\label{eqq1.1'}
	\left\{
	\begin{array}{lll}
\displaystyle{
\rho_1\varphi_{tt}-[k_1 \left(\varphi_x+\psi+\ell  w\right)+D_1\left(\varphi_{xt}+\psi_t+\ell  w_t\right)]_x-\ell  k_3\left(w_x-\ell \varphi\right)-\ell D_3\left(w_{xt}-\ell \varphi_t\right)=0,}\\ \\
\displaystyle{\rho_2 \psi_{tt}-[k_2\psi_{x}+D_2\psi_{xt}]_x+k_1\left(\varphi_x+\psi+\ell w\right)+D_1\left(\varphi_{xt}+\psi_t+\ell  w_t\right)=0,}\\ \\
\displaystyle{\rho_1w_{tt}-[k_3\left(w_x-\ell \varphi\right)
+D_3\left(w_{xt}-\ell \varphi_t\right)]_x
+\ell k_1\left(\varphi_x+\psi+\ell w\right)+\ell D_1\left(\varphi_{xt}+\psi_t+\ell  w_t\right)=0,}
	\end{array}
	\right.
\end{equation}
with fully Dirichlet boundary conditions:
\begin{equation}
	\begin{array}{lll}\label{DDDD}
\varphi\left(0,\cdot\right)=\varphi\left(L,\cdot\right)=\psi\left(0,\cdot\right)=\psi\left(L,\cdot\right)=
w\left(0,\cdot\right)=w\left(L,\cdot\right)=0\quad&\text{in }\mathbb{R}_{+},
	\end{array}
\end{equation}
or  with Dirichlet-Neumann-Neumann boundary conditions:
\begin{equation}\label{DNND}
	\begin{array}{lll}
\varphi\left(0,\cdot\right)=\varphi\left(L,\cdot\right)=\psi_x\left(0,\cdot\right)=\psi_x\left(L,\cdot\right)=
w_x\left(0,\cdot\right)=w_x\left(L,\cdot\right)=0\quad&\text{in }\mathbb{R}_{+},
	\end{array}
\end{equation}
in addition to  the following initial conditions:
\begin{equation}\label{initial conditions}
	\begin{array}{lll}
\varphi\left(\cdot,0\right)=\varphi_0\left(\cdot\right),\ {\psi\left(\cdot,0\right)=\psi_0\left(\cdot\right)},\

w\left(\cdot,0\right)=w_0\left(\cdot\right),\\  \noalign{\medskip} 	\varphi_t\left(\cdot,0\right)=\varphi_1\left(\cdot\right),\
\psi_t\left(\cdot,0\right)=\psi_1\left(\cdot\right),\
w_t\left(\cdot,0\right)=w_1\left(\cdot\right), \ \ \ \text{in} \ (0,L).\\
	\end{array}
\end{equation}
%%%%%%%%%%%%%%%%%
% New SG
%%%%%%%%%%%%%%%%%
We define the three wave speeds  as: 
\[
c_{1} = \sqrt{\dfrac{k_{1}}{\rho_{1}}} \quad , \quad c_{2} = \sqrt{\dfrac{k_{2}}{\rho_{2}}}\quad,\quad c_{3} = \sqrt{\dfrac{k_{3}}{\rho_{1}}} \quad .
\]
In the absence of the three Kelvin-Voigt damping terms, the system \eqref{eqq1.1'} is a system of three coupled wave equations. 
This system is conservative whereas when at least one of the three Kelvin-Voigt damping  is present, the system is dissipative. The combination of direct damping,
that is damping that acts in the equation involving the unknown itself and indirect damping that acts on another unknown than the one concerns by the equation,
makes this study much delicate.

We note that when $R\rightarrow\infty$, then $\ell \rightarrow0$ and the Bresse model reduces, by neglecting $w$, to the well-known Timoshenko beam equations:
\begin{equation}\label{Tim}
	\left\{
	\begin{array}{lll}
\displaystyle{
\rho_1\varphi_{tt}-[k_1 \left(\varphi_x+\psi \right)+D_1\left(\varphi_{xt}+\psi_t \right)]_x=0,}\\ \\
\displaystyle{\rho_2 \psi_{tt}-[k_2\psi_{x}+D_2\psi_{xt}]_x+k_1\left(\varphi_x+\psi\right)+D_1\left(\varphi_{xt}+\psi_t\right)=0}
	\end{array}
	\right.
\end{equation}
with different types of boundary conditions and with initial data. 

%%%%%%%%%%%%
\subsection{Motivation, aims and main results} \label{Aims} The stability of elastic Bresse system with different types of damping (frictional, thermoelastic, Cattaneo, ...)
has been intensively studied (see Subsection \ref{ssb}), but there are a few number of papers concerning the stability of Bresse or Timoshenko systems with local 
viscoelastic Kelvin-Voigt damping. 
In fact, in \cite{ArwadeYoussef}, El Arwadi and Youssef studied the theoretical and numerical stability on a Bresse system with Kelvin-Voigt damping under fully Dirichlet boundary conditions. 
Using multiplier techniques,  they established an exponential energy decay rate provided that the system is subject to three global Kelvin-Voigt damping. 
Later, a numerical scheme based on the finite element method was introduced to approximate the solution. Zhao \textit{et al.} in \cite{ZhaoLiuZhang}, 
considered a Timoshenko system with Dirichlet-Neumann boundary conditions. 
They obtained the exponential stability under certain hypotheses of the smoothness and structural condition of the coefficients of the system, 
and obtain the strong asymptotic stability under weaker hypotheses of the coefficients. Tian and Zhang in \cite{Tian2017} 
considered a Timoshenko system under fully Dirichlet boundary conditions and with two locally or globally Kelvin-Voigt dampings.
First, in the case when the two Kelvin-Voigt dampings are globally distributed, they  showed that the corresponding semigroup is analytic.
On the contrary, they proved that the energy of the system decays exponentially or polynomially and the decay rate depends on properties of material
coefficient function. In \cite{GhaderWehbe}, Ghader and Wehbe generalized the results of \cite{ZhaoLiuZhang}  and \cite{Tian2017}.
Indeed, they  considered the Timoshenko system with only one locally or globally distributed Kelvin-Voigt damping and subject to fully Dirichlet  or to Dirichlet-Neumann boundary conditions.
They established a polynomial energy decay rate of type $t^{-1}$ for smooth initial data.
Moreover, they proved that the obtained energy decay rate is in some sense optimal.
In \cite{Maryati}, Maryati \textit{et al.} considered the transmission problem of a Timoshenko beam composed by $N$ components, 
each of them being either purely elastic, or a Kelvin-Voigt viscoelastic material, or an elastic material inserted with a frictional damping mechanism.
They proved that the energy decay rate depends on the position of each component. In particular, they proved that the model is exponentially stable if and only
if all the elastic components are connected with one component with frictional damping. Otherwise, only a polynomial energy decay rate is established.
So, the stability of the Bresse system with local viscoelastic Kelvin-Voigt damping  is still an open problem.

The purpose of this paper is to study the Bresse system in the presence of  local
non-smooth dampings coefficient at interface and under fully Dirichlet or Dirichlet-Neumann-Neumann boundary conditions.
The system is given by \eqref{eqq1.1'}-\eqref{DDDD} or \eqref{eqq1.1'}-\eqref{DNND} with initial data \eqref{initial conditions}. 

When $D_1$, $D_2$, $D_3 \in L^\infty(0,L)$, using frequency domain approach combined with multiplier techniques and the construction of new multiplier functions,
we establish a polynomial stability of type $\dfrac{1}{t}$ (see Theorem \ref{PolyThe.2}). 
Moreover,  in the presence of only one local damping $D_2$ acting on the shear angle displacement ($D_1=D_3=0$), we establish a polynomial energy decay estimate of type 
$\dfrac{1}{\sqrt{t}}$ (see Theorem \ref{PolyThe1.2}).

Finally, in the absence of at least one damping, we prove the lack of uniform stability for the system \eqref{eqq1.1'}-\eqref{DNND} even 
with smoothness of damping coefficients.  In these cases, we conjecture the optimality of the obtained decay rate. For clarity, let 
\[
\emptyset \not=\omega=(\alpha, \beta) \subset (0,L).
\]
Here and thereafter, $\alpha$ and $\beta$ will be considered as interfaces.

%%%%%%%%%%%%%%%%%%%%%%%%%%%%%%%%%%%%%%%%%%%%%%%%%%%%%%%%%%%%%%%%%%
\subsection{Literature concerning the Bresse system} \label{ssb} In  \cite{LR4}, Liu and Rao considered the Bresse system with two thermal 
dissipation laws. They proved an exponential decay rate when the wave speed of the vertical displacement coincides with the wave speed of 
longitudinal displacement or of the shear angle displacement. Otherwise, they showed polynomial decays depending on the boundary conditions.
These results are improved by Fatori and Rivera in \cite{RiveraBresse} where they considered the case of one thermal dissipation law globally distributed 
on the displacement equation. Wehbe and Najdi in \cite{Wehbenadine} extended and improved the results of \cite{RiveraBresse}, when the thermal 
dissipation is locally distributed. Wehbe and Youssef in \cite{wehbey} considered an elastic Bresse system subject to two locally internal dissipation laws.
They proved that the system is exponentially stable if and only if the waves propagate at the same speed. Otherwise, a polynomial decay holds.
Alabau \textit{et al.} in \cite{AlabauBresse} considered the same system with one globally distributed dissipation law.
The authors proved the existence of polynomial decays with rates that depend on some particular relation between the coefficients.
In \cite{aissaguesmia}, Guesmia \textit{et al.}  considered Bresse system with infinite memories acting in the three equations of the system.
They established  asymptotic stability results under some conditions on the relaxation functions regardless the speeds of propagation.
These results are improved by Abdallah \textit{et al.} in \cite{Ghader} where they considered the Bresse system with infinite memory type control 
and/or with heat conduction given by Cattaneo's law acting in the shear angle displacement. The authors established an exponential energy decay
rate when the waves propagate at same speed.  Otherwise, they showed polynomial decays. 
In \cite{Benaissa}, Benaissa and Kasmi, considered the Bresse system with three control of  fractional derivative type acting on the boundary conditions. 
They established a polynomial decay estimate. \\

\subsection{Organization of the paper} This paper is organized as follows: In Section   \ref{se2}, we prove the well-posedness  of system
\eqref{eqq1.1'} with either the boundary conditions \eqref{DDDD} or \eqref{DNND}. Next, in Section \ref{strong},
we prove the strong stability of the system in the lack of the compactness of the resolvent of the generator. 

In Section \ref{Polynomial1} when the coefficient functions $D_1$, $D_2$, and $D_3$ are not smooth, we prove the polynomial stability of type $\frac{1}{t}$. 
In section \ref{Polynomial2}, we prove the polynomial energy decay rate of type $\frac{1}{\sqrt{t}}$ for the system 
in the case of {only} one local non-smooth damping $D_2$ acting on the shear angle displacement.  
In Section \ref{NonUni}, under boundary conditions \eqref{DNND}, we prove the lack of uniform (exponential)
stability of the system in the absence of at least one damping.
Finally in Section \ref{additional}, we will briefly state the analytic stabilization of the system \eqref{eqq1.1'} when the three damping coefficient act on the whole spatial
domain $(0,L)$ and the exponential stability when the three damping coefficient are localized on $(\alpha,\beta)$ and are smooth at the interfaces. 
%%%%%%%%%%%%%%%%%%%%%%%%%%%%%%%%%%%%%%%%%
\section{Well-posedness of the problem} \label{se2}
In this part, using a semigroup approach, we establish the well-posedness result for the systems \eqref{eqq1.1'}-\eqref{DDDD} and \eqref{eqq1.1'}-\eqref{DNND}.
Let $(\varphi,\psi,w)$ be a regular solution of system \eqref{eqq1.1'}-\eqref{DDDD}, its associated energy is given by:
\begin{equation}\label{ Equation 2.1}
	\begin{array}{ll}
\displaystyle{E\left(t\right)}= \displaystyle{\frac{1}{2}\bigg\{\int_0^L\left(\rho_1\left|\varphi_{t}\right|^{2}+\rho_2\left|\psi_t\right|^2+\rho_1\left|w_t\right|^{2}+k_1\left|\varphi_x+\psi+\ell  w\right|^{2}\right)dx}
\\  \noalign{\medskip}  \hspace{1.6cm} \displaystyle{+\int_0^L \left({k}_2\left|\psi_x\right|^{2}+k_3\left|w_x-\ell  \varphi\right|^{2}\right) dx}\bigg\},
	\end{array}
\end{equation}
and it is dissipated according to the following law:
\begin{equation}\label{ Equation 2.2}
	E'\left(t\right)=-\int_0^L \left(D_1|\varphi_{xt}+\psi_t +\ell w_t|^2 +D_2|\psi_{xt}|^2 + D_3|w_{xt}-\ell \varphi_t|^2\right)dx\leq 0.
\end{equation}
Now, we define the following energy spaces:
$$\HH_1= \left(H_0^1(0,L) \times L^2(0,L)\right)^3 \ \ \mathrm{and} \ \ \HH_2=H_0^1(0,L) \times L^2 (0,L)\times\left(  H_{*}^1(0,L) \times L^2_{*}(0,L)\right)^2,$$
%%%%%%%%%%%%
where
$$L_*^2(0,L)=\{f\in L^2(0,L):\int_0^Lf(x)dx=0\}\,\,{\rm and}\,\,
H_*^1(0,L)=\{f\in H^1(0,L):\int_0^Lf(x)dx=0\}.$$
%%%%%%%%%%%%%%%%
Both spaces $\HH_1$ and $\HH_2$ are equipped with the inner product which induces the
energy norm:
\begin{equation}\label{ Equation 2.3}
	\begin{array}{ll}
\displaystyle{\|U\|_{\mathcal{H}_j}^2}&= \displaystyle{\|(v^1,v^2,v^3,v^4,v^5,v^6)\|_{\mathcal{H}_j}^{2}}\\  
&=\rho_1\left\|v^2\right\|^{2}+\rho_2\left\|v^4\right\|^2+\rho_1\left\|v^6\right\|^{2}
+k_1\left\|v^1_x+v^3+\ell  v^5\right\|^{2} \\  \noalign{\medskip}
&+{k}_2\left\|v^3_x\right\|^{2}+k_3\left\|v^5_x-\ell  v^1\right\|^{2}, \ \ j=1,2\
	\end{array}
\end{equation}
here and after $\|\cdot\|$ denotes  the norm of $L^2\left(0,L\right)$ . 
%%%%%%%%%%%%%%%%%%%%%%%%%%%%%%%%%%%%%%%%%%
\begin{rk}
	In the case of boundary condition (\ref{DDDD}), it is easy to see that expression (\ref{ Equation 2.3})
	defines a norm on the energy space $\HH_1$. But in the case of boundary
	condition (\ref{DNND}) the expression (\ref{ Equation 2.3}) define a norm on the energy space $\HH_2$ if $L\neq\dfrac{n\pi}{\ell }$
	for all positive integer $n$. Then, here and after, we assume that there does not exist any $n \in \nn$
	such that $L=\dfrac{n\pi}{\ell }$ when $j = 2$.
\end{rk}
%%%%%%%%%%%%%%%%%%%%%%%%%%%%%%%%%%%%%%%%%%
Next, we define the linear operator $\mathcal{A}_j$ in $\mathcal{H}_j$ by:
\begin{equation*}
	\begin{array}{l}
D\left(\mathcal{A}_1\right)=\bigg\{\ U\in\mathcal{H}_1 \ |\  v^2,v^4, v^6 \in H_0^1\left(0,L\right),\  \left[k_1
\left(v_x^1+v^3+\ell v^5\right) + D_1\left(v_x^2+v^4+\ell v^6\right)\right]_x\in L^2(0,L)\,, \\ \hspace{3cm} \ 
\left[k_2v_x^3+D_2v_x^4\right]_x \in L^2(0,L), \, \left[k_3(v_x^5-\ell v^1)+D_3(v_x^6-\ell v^2)\right]_x \in L^2(0,L)\bigg\},
	\end{array}
\end{equation*}
%%%%%%%%%%%%%%%%%%%%%%%%%%%%%%%%%%%%%%%%%%%%%%%%%%
\begin{equation*}
	\begin{array}{l}
D\left(\mathcal{A}_2\right)=\bigg\{\ U\in\mathcal{H}_2 \ |\  v^2 \in H_0^1\left(0,L\right), v^4,v^6 \in H_*^1\left(0,L\right),v^3_x|_{0,L}=v^5_x|_{0,L}=0 ,\\ \hspace{3cm} \  \left[k_1
\left(v_x^1+v^3+\ell v^5\right) + D_1\left(v_x^2+v^4+\ell v^6\right)\right]_x\in L^2(0,L)\,, \\ \hspace{3cm} \ 
\left[k_2v_x^3+D_2v_x^4\right]_x \in L_{*}^2(0,L), \, \left[k_3(v_x^5-\ell v^1)+D_3(v_x^6-\ell v^2)\right]_x \in L_{*}^2(0,L)\bigg\}
	\end{array}
\end{equation*}
%%%%%%%%%%%%%%%%%%%%%%%%%%%%%%%%%%%%%%%%%%%%%%%%%
and
\begin{equation}\label{ Equation 2.4}
	\mathcal{A}_j\left(\begin{array}{l}
v^1\\  \noalign{\medskip} v^2\\  \noalign{\medskip}v^3\\ v^4\\  \noalign{\medskip} v^5\\  \noalign{\medskip} v^6
	\end{array}\right)=\left(\begin{array}{c}
v^2\\  \noalign{\medskip} \rho_1^{-1}\left(\left[k_1(v_x^1+v^3+\ell v^5) + D_1(v_x^2+v^4+\ell v^6)\right]_x  +\ell k_3(v_x^5-\ell v^1)+\ell D_3(v_x^6-\ell v^2)\right)\\
\noalign{\medskip} v^4\\  \noalign{\medskip} 
\rho_2^{-1} \left((k_2v_x^3+D_2v_x^4)_x-k_1\left(v^1_x+v^3+\ell  v^5\right)-D_1(v_x^2+v^4+\ell v^6)\right)\\  \noalign{\medskip}
v^6\\  \noalign{\medskip}
\rho_1^{-1}\left(\left[k_3(v^5_x-\ell  v^1)+D_3(v_x^6 -\ell v^2)\right]_x-\ell  k_1\left(v^1_x+v^3+\ell  v^5\right) -\ell  D_1(v_x^2+v^4+\ell v^6)\right)
	\end{array}\right)
\end{equation}
%%%%%%%%%%%%%%%%%%%%%%%%%%%%%%%%%%%%%%%%%%%%%%%%%%
for all $U=\left(v^1,v^2,v^3,v^4,v^5,v^6\right)^{\mathsf{T}}\in D\left(\mathcal{A}_j\right)$. So, if $U=\left(\varphi,\varphi_t,\psi,\psi_t,w,w_t\right)^{\mathsf{T}}$
is the state of \eqref{eqq1.1'}-\eqref{DDDD} or \eqref{eqq1.1'}-\eqref{DNND}, then the Bresse beam system is transformed into a first
order evolution equation on the Hilbert space $\mathcal{H}_j$:
\begin{equation}\label{Bresse Equation 2.11}
	\left\{
	\begin{array}{c}
U_t=\mathcal{A}_jU, \,\, j=1,2\\  \noalign{\medskip}
U\left(0\right)=U_0(x),
	\end{array}
	\right.
\end{equation}
where
\begin{equation*}
	U_0\left(x\right)=\left(\varphi_0\left(x\right),\varphi_1\left(x\right),{\psi_0\left(x\right)},\psi_1\left(x\right),w_0\left(x\right),
	w_1\left(x\right)\right)^{\mathsf{T}}.
\end{equation*}  
%%%%%%%%%%%%%%%%%%%%%%%%%%%%%%%%%%%%%%%%%%%%%%%%%%
% Remark 2.2%
%%%%%%%%%%%%%%%%%%%%%%%%%%%%%%%%%%%%%%%%%%%%%%%%%%
\begin{rk}\label{Remark 2.2}
	It is easy to see that  there exists a positive constant $c_0$ such that { for any $\left(\varphi,\psi,w\right)\in \left(H_0^1\left(0,L\right)\right)^3$ for $j=1$ 
and for any $\left(\varphi,\psi,w\right)\in H_0^1\left(0,L\right)\times \left(H_*^1\left(0,L\right)\right)^2$ for $j=2$,}
	\begin{equation}\label{ Equation 2.5}
k_1\left\|\varphi_x+\psi+\ell w\right\|^2+{k}_2\left\|\psi_x\right\|^2+k_3\left\|w_x-\ell \varphi\right\|^2\leq 
c_0\left(\left\|\varphi_x\right\|^2+\left\|\psi_x\right\|^2+\left\|w_x\right\|^2\right).
	\end{equation}
	%%%%%%%%%%%%%%%%%%%%%%%%%%%%%%%%%%%%%%%%%%%%%%%%%%
	On the other hand, we can show  by a contradiction argument the existence of a positive constant $c_1$  such that, for any 
$\left(\varphi,\psi,w\right)\in \left(H_0^1\left(0,L\right)\right)^3$ for $j=1$ and for any $\left(\varphi,\psi,w\right)\in H_0^1\left(0,L\right)\times \left(H_*^1\left(0,L\right)\right)^2$ for $j=2$,
	\begin{equation}\label{Bresse Equation 2.6}
c_1\left(\left\|\varphi_x\right\|^2+\left\|\psi_x\right\|^2+\left\|w_x\right\|^2\right)\leq 
k_1\left\|\varphi_x+\psi+\ell w\right\|^2+ {k}_2\left\|\psi_x\right\|^2+k_3\left\|w_x-\ell \varphi\right\|^2.
	\end{equation}
	%%%%%%%%%%%%%%%%%%%%%%%%%%%%%%%%%%%%%%%%%%%%%%%%%%
	Therefore the norm on the energy space $\mathcal{H}_j$ given in \eqref{ Equation 2.3} is equivalent to the usual norm on $\mathcal{H}_j$.
\end{rk}
%%%%%%%%%%%%%%%%%%%%%%%%%%%%%%%%%%%%%%%%%%%%%%%%%%
% Proposition 2.3%
%%%%%%%%%%%%%%%%%%%%%%%%%%%%%%%%%%%%%%%%%%%%%%%%%%
\begin{prop}\label{Bresse Proposition 2.3}
	Assume that coefficients functions $D_1$, $D_2$ and $D_3$ are non negative. Then, the operator $\mathcal{A}_j$ is m-dissipative in the energy space $\mathcal{H}_j$, for $j=1,2$.
\end{prop}
%%%%%%%%%%%%%%%%%%%%%%%%%%%%%%%%%%%%%%%%%%%%%%%%%%
% Proof of Proposition 2.3%
%%%%%%%%%%%%%%%%%%%%%%%%%%%%%%%%%%%%%%%%%%%%%%%%%%
\rm{\begin{proof}
Let $U=\left(v^1,v^2,v^3,v^4,v^5,v^6\right)^{\mathsf{T}}\in D\left(\mathcal{A}_j\right)$. {By} a {straightforward} calculation, we have:
\begin{equation}\label{ Equation 2.14}
\mathrm{Re}\left(\mathcal{A}_jU,U\right)_{\mathcal{H}_j}=-
\int_0^{L}\left( D_1\left|v_x^2+v^4+\ell v^6\right|^2 + D_2\left|v_x^4\right|^2 +D_3\left|v_x^6-\ell v^2\right|^2 \right)dx.
\end{equation}
%%%%%%%%%%%%%%%%%%%%%%%%%%%%%%%%%%%%%%%%%%%%%%%%%%
As $D_1\geq 0, \, D_2 \geq 0 \,\,\mathrm{and} \,  D_3 \geq 0$ , we get that  $\mathcal{A}_j$ is  dissipative.\\

Now, we will check the maximality of $\AA_j$. For this purpose, let  $F=\left(f^1,f^2,f^3,f^4,f^5,f^6\right)^\mathsf{T}\in\mathcal{H}_j,$ 
we have  to prove the existence of 
$U=\left(v^1,v^2,v^3,v^4,v^5,v^6\right)^\mathsf{T}\in D\left(\mathcal{A}_j\right)$ 
unique solution of the equation
$-\mathcal{A}_jU=F$.

Let $\left(\varphi^1, \varphi^3,\varphi^5\right) \in \left( H_0^1(0,L)\right)^3$ for $j=1$ and   
$\left(\varphi^1, \varphi^3,\varphi^5\right) \in \left(H_0^1(0,L)\times (H_{*}^1(0,L))^2 \right)$ for $j=2$ be a test function.
Writing $-\mathcal{A}_{j} U$ and replacing the first, third and fith component of  $U$ by $-f^1\,,\, -f^3 \,,\,-f^5$ and now multiplying the second, the fourth and the sixth equation by 
respectively $\varphi^1, \varphi^3,\varphi^5$, after integrating by parts, we obtain the following form:
\begin{equation}\label{Bresse Equation 2.24}
\left\{
\begin{array}{ll}
	\displaystyle{k_1 \left(v^1_x+v^3+\ell  v^5\right)\varphi^1_x-\ell  k_3\left(v^5_x-\ell  v^1\right)\varphi^1
=h^1,}
	\\ \noalign{\medskip}
	\displaystyle{k_2 v^3_{x}\varphi_x^3+k_1\left(v^1_x+v^3+\ell  v^5\right) \varphi^3=h^3},
	\\ \noalign{\medskip}
	\displaystyle{k_3\left(v^5_x-\ell  v^1\right) \varphi_x^5+\ell  k_1\left(v^1_x+v^3+\ell v^5\right)\varphi^5=h^5,}
\end{array}
\right.
\end{equation}
where  
\begin{align*}
h^1&=\rho_1 f^2\varphi^1 + D_1\left(f_x^1+f^3+\ell f^5\right)\varphi_x^1 - \ell D_3(f_x^5 -\ell f^1)\varphi^1, \\
h^3&=\rho_2 f^4\varphi^3 +D_2 f_x^3\varphi_x^3+ D_1\left(f_x^1+f^3+\ell f^5\right)\varphi^3, \mbox{ and} \\
h^5&=\rho_1 f^6\varphi^5 +D_3\left(f^5-\ell f^1\right)\varphi_x^5+ \ell D_1\left(f_x^1+f^3+\ell f^5\right)\varphi^5.
\end{align*}
Using Lax-Milgram Theorem (see \cite{Pazy83}), we deduce that \eqref{Bresse Equation 2.24} admits a unique solution in 
$\left( H^1_0\left(0,L\right)\right)^3$ for $j=1$ and in $\left(H_0^1(0,L)\times (H_{*}^1(0,L))^2 \right)$ for $j=2$.
Thus, $-\mathcal{A}_jU=F$ admits {an} unique solution $U\in D\left(\mathcal{A}_j\right)$ and consequently $0\in \rho(\mathcal{A}_j)$.
Then, $\mathcal{A}_j$ is closed and consequently $\rho\left(\mathcal{A}_j\right)$ is open set of $\mathbb{C}$ (see Theorem 6.7 in \cite{Kato76}).
Hence,  we easily get $R(\lambda I -\mathcal{A}_j ) = \mathcal{H}_j$ for sufficiently small $\lambda>0 $. This, together with the
dissipativeness of $\mathcal{A}_j$, imply that   $D\left(\mathcal{A}_j\right)$ is dense in $\mathcal{H}_j$   and that $\mathcal{A}_j$
is m-dissipative in $\mathcal{H}_j$ (see Theorems 4.5, 4.6 in  \cite{Pazy83}).
The proof is thus complete.

\end{proof}}
%%%%%%%%%%%%%%%%%%%%%%%%%%%%%%%%%%%%%%%%%%%%%%%%%% 
%%%%%%%%%%%%%%%%%%%%%%%%%%%%%%%%%%%%%%%%%%%%%%%%%% 
%%%%%%%%%%%%%%%%%%%%%%%%%%%%%%%%%%%%%%%%%%%%%%%%%% 
\noindent	Thanks to  Lumer-Phillips Theorem (see \cite{LiuZheng01, Pazy83}), we deduce that $\mathcal{A}_j$ generates a  $C_0$-semigroup of contraction $e^{t\mathcal{A}_j}$ in $\mathcal{H}_j$ and therefore  problem \eqref{Bresse Equation 2.11} is well-posed. We have thus the following result.
%%%%%%%%%%%%%%%%%%%%%%%%%%%%%%%%%%%%%%%%%%%%%%%%%%
% Theorem 2.4%
%%%%%%%%%%%%%%%%%%%%%%%%%%%%%%%%%%%%%%%%%%%%%%%%%%
\begin{thm}\label{Theorem 2.4}
	For any $U_0\in\mathcal{H}_j$, problem \eqref{Bresse Equation 2.11}  admits a unique weak solution
	$$U\in C\left(\mathbb{R}_{+};\mathcal{H}_j\right).$$
	%%%%%%%%%%%%%%%%%%%%%%%%%%%%%%%%%%%%%%%%%%%%%%%%%%
	Moreover, if $U_0\in D\left(\mathcal{A}_j\right),$ then
	$$U\in C\left(\mathbb{R}_{+};D\left(\mathcal{A}_j\right)\right) \cap C^1\left(\mathbb{R}_{+};\mathcal{H}_j\right).$$
\end{thm}
%%%%%%%%%%%%%%%%%%%%%%%%%%%%%%%%%%%%%%%%%%%%%%%%%%
%%%%%%%%%%%%%%%%%%%%%%%%%%%%%%%%%%%%%%%%%%%%%%%%%%
%%%%%%%%%%%%%%%%%%%%%%%%%%%%%%%%%%%%%%%%%%%%%%%%%%
% Section 2.2: Strong stability %
%%%%%%%%%%%%%%%%%%%%%%%%%%%%%%%%%%%%%%%%%%%%%%%%%%
%%%%%%%%%%%%%%%%%%%%%%%%%%%%%%%%%%%%%%%%%%%%%%%%%%	
%%%%%%%%%%%%%%%%%%%%%%%%%%%%%%%%%%%%%%%%%%%%%%%%%%
\section{Strong stability of the system}\label{strong}
In this part, we use a general criteria of Arendt-Batty in \cite{Arendt-Batty88}   to show the strong stability of the $C_0$-semigroup $e^{t\mathcal{A}_j}$ associated to the Bresse system \eqref{eqq1.1'} in the absence of the compactness of the resolvent of $\mathcal{A}_j$. Before, we state our main result, we need the following stability condition:
\begin{equation*}
	\mbox{(SSC)} \quad \quad
	\mbox{There exist}\, i \in  \{1,2,3\},\, d_0>0\, \,\mathrm{and}\,\, \alpha<\beta \in [0,L] \, \mbox{such that}\,\, D_i \geq d_0 >0\, \,\mathrm{on} \,\, (\alpha, \beta).
\end{equation*}
%%%%%%%%%%%%%%%%%%%%%%%%%%%%%%%%%%%%%%%%%%%%%%%%%%
% Theorem 2.5%
%%%%%%%%%%%%%%%%%%%%%%%%%%%%%%%%%%%%%%%%%%%%%%%%%%
\begin{thm}\label{Bresse Theorem 2.5}
	Assume that condition (SSC) holds. Then the $C_0-$semigroup $e^{t\mathcal{A}_j}$ is strongly stable in $\mathcal{H}_j$, $j=1, 2$, \textit{i.e.}, for all $U_0\in\mathcal{H}_j$, the solution of \eqref{Bresse Equation 2.11} satisfies
	%%%%%%%%%%%%%%
	$$
	\lim_{t\to+\infty}\left\|e^{t\mathcal{A}_j} U_0\right\|_{\mathcal{H}_j}=0.
	$$
\end{thm}
%%%%%%%%%%%%%%%%%%%%%%%%%%%%%%%%%%%%%%%%%%%%%%%
%%%%%%%%%%%%%%%%%%%%%%%%%%%%%%%%%%%%%%%%%%%%%%%
%%%%%%%%%%%%%%%%%%%%%%%%%%%%%%%%%%%%%%%%%%%%%%%
For the proof of Theorem  \ref{Bresse Theorem 2.5}, we need the following two lemmas.
%%%%%%%%%%%%%%%%%%%%%%%%%%%%%%%%%%%%%%%%%%%%%%%%%%
% Lemma 2.6%
%%%%%%%%%%%%%%%%%%%%%%%%%%%%%%%%%%%%%%%%%%%%%%%%%%
\begin{lem}\label{Bresse Lemma 2.6}
	Under the same condition of Theorem \ref{Bresse Theorem 2.5}, we have
	\begin{equation}\label{Bresse Equation 2.25}
\ker\left(\ i\lambda-\mathcal{A}_j\right)=\{0\},\,\,j=1,2,\,\, \, {\rm for\, all } \,\, \lambda\in\mathbb{R}.
	\end{equation}
\end{lem}
%%%%%%%%%%%%%%%%%%%%%%%%%%%%%%%%%%%%%%%%%%%%%%%%%%
% Proof of Lemma 2.6%
%%%%%%%%%%%%%%%%%%%%%%%%%%%%%%%%%%%%%%%%%%%%%%%%%%
\begin{proof} We will prove Lemma \ref{Bresse Lemma 2.6} in the case $D_1=D_3=0$ on $(0,L)$ and $D_2 \geq d_0 >0$ on $(\alpha, \beta) \subset (0,L)$. The other cases are similar to prove. \\
	First, from  Proposition \ref{Bresse Proposition 2.3}, we claim that $0\in\rho\left(\mathcal{A}_j\right).$  We still have to show the result for 
	$\lambda\in\mathbb{R^*}$. Suppose that there exist a real number $\lambda\neq 0$  and
	$ 0 \neq U=\left(v^1,v^2,v^3,v^4,v^5,v^6\right)^{\mathsf{T}}\in D\left(\mathcal{A}_j\right)$ such that:
	\begin{equation}\label{Bresse Equation 2.26}
\mathcal{A}_jU=\ i\lambda U.
	\end{equation}
	%%%%%%%%%%%%%%%%%%%%%%%%%%%%%%%%%%%%%%%%%%%%%%%%%%
	Our goal is to find a contradiction by proving that $U=0$. Taking the real part of the inner product in $\HH_j$  of $\AA _jU$  and $U$, we get:
	\begin{equation}\label{Sam1}
\mathrm{Re}\left(\mathcal{A}_jU,U\right)_{\mathcal{H}_j}= - \int_0^L D_2\left|v_x^4\right|^2dx =0.
	\end{equation}
	%%%%%%%%%%%%%%%%%%%%%%%%%%%%%%%%%%%%%%%%%%%%%%%%%%
	Since by assumption $D_2 \geq d_0 >0$ on $(\alpha, \beta)$,  it follows from equality \eqref{Sam1} that:
	\begin{equation}\label{dissipation}
D_2v^4_x=0 \quad \mathrm{in} \quad (0, L)\, \, \mathrm{and} \, \, v^4_x=0 \quad \mathrm{in} \quad (\alpha, \beta). 
	\end{equation}
	%%%%%%%%%%%%%%%%%%%%%%%%%%%%%%%%%%%%%%%%%%%%%%%%%%
	Detailing \eqref{Bresse Equation 2.26} we get:
	\begin{eqnarray}
v^2&=&\ i\lambda v^1, \label{Bresse Equation 2.27}
\\
k_1 \left(v^1_x+v^3+\ell  v^5\right)_x+\ell  k_3\left(v^5_x-\ell  v^1\right)&=&\ i\rho_1\lambda v^2,\label{Bresse Equation 2.28}
\\
v^4&=&\ i\lambda v^3,\label{Bresse Equation 2.29}
\\
\left(k_2v^3_x+D_2v^4_x\right)_x-k_1\left(v^1_x+v^3+\ell  v^5\right)
&=&\ i\rho_2\lambda v^4,\label{Bresse Equation 2.30}
\\
v^6&=&\ i\lambda v^5,\label{Bresse Equation 2.31}
\\
k_3\left(v^5_x-\ell v^1\right)_x-\ell k_1\left(v^1_x+v^3+\ell  v^5\right)&=&\ i\rho_1\lambda v^6.\label{Bresse Equation 2.32}
	\end{eqnarray}
	%%%%%%%%%%%%%%%%%%%%%%%%%%%%%%%%%%%%%%%%%%%%%%%%%%
	Next, inserting  \eqref{dissipation} in \eqref{Bresse Equation 2.29} and using the fact that $\lambda \neq0$, we get:
	\begin{equation}\label{dissipation2}
v^3_x=0 \quad \mathrm{in} \quad (\alpha, \beta). 
	\end{equation}
	%%%%%%%%%%%%%%%%%%%%%%%%%%%%%%%%%%%%%%%%%%%%%%%%
	Moreover, substituting equations \eqref{Bresse Equation 2.27}, \eqref{Bresse Equation 2.29} and \eqref{Bresse Equation 2.31} into equations 
\eqref{Bresse Equation 2.28}, \eqref{Bresse Equation 2.30} and \eqref{Bresse Equation 2.32}, we get:
	\begin{equation} \label{Bresse strong}
\left\{
\begin{array}{ll}
\rho_1 \lambda^2 v^1+k_1 \left(v^1_x+v^3+\ell  v^5\right)_x+\ell  k_3\left(v^5_x-\ell  v^1\right)=0, \\ \noalign{\medskip}
\rho_2 \lambda^2 v^3 +\left(k_2v^3_x+i D_2 \lambda v^3_x\right)_x-k_1\left(v^1_x+v^3+\ell  v^5\right)=0,\\ \noalign{\medskip}
\rho_1 \lambda^2 v^5+k_3\left(v^5_x-\ell v^1\right)_x-\ell k_1\left(v^1_x+v^3+\ell  v^5\right)=0.
\end{array}
\right.
	\end{equation}
	%%%%%%%%%%%%%%%%%%%%%%%%%%%%%%%%%%%%%%%%%%%%%
	{Now, we introduce the functions $\widehat{v}^{i}$, for $i=1,..,6$ by} 
	$
	\widehat{v}^{i}=v_{x}^{i} \ .
	$
	It is easy to see that $\widehat{v}^{i} \in H^1(0,L)$.\\
	%let $\left(\widehat v^1, \widehat v^2, \widehat v^3, \widehat v^4, \widehat v^5, \widehat v^6\right) = \left( v_x^1, v^2_x, v^3_x, v^4_x, v^5_x, v^6_x \right)$.  
	It follows from equations
	\eqref{dissipation} and \eqref{dissipation2} that:
	\begin{equation}\label{Bresse strong.1}
\widehat v^3 = \widehat v^4 =0 \quad \mathrm{in} \quad (\alpha, \beta)
	\end{equation}
	and consequently system \eqref{Bresse strong} will be, after differentiating it with respect to $x$, given by:  
	\begin{eqnarray} 
\rho_1 \lambda^2 \widehat v^1+k_1 \left(\widehat v^1_x+\ell  \widehat v^5\right)_x+\ell  k_3\left(\widehat v^5_x-\ell  \widehat v^1\right)&=&0 \quad \mathrm{in} \quad (\alpha, \beta),\label {Bresse Equation 2.33}\\ \noalign{\medskip}
\widehat v^1_x+\ell  \widehat v^5&=&0   \quad \mathrm{in} \quad (\alpha, \beta),\label {Bresse Equation 2.34}\\ \noalign{\medskip}
\rho_1 \lambda^2 \widehat v^5+k_3\left(\widehat v^5_x-\ell  \widehat v^1\right)_x-\ell k_1\left(\widehat v^1_x+\ell  \widehat v^5\right)&=&0 \quad \mathrm{in} \quad (\alpha, \beta). \label {Bresse Equation 2.35}
	\end{eqnarray}
	Furthermore, substituting equation \eqref{Bresse Equation 2.34} into  
	\eqref{Bresse Equation 2.33} and \eqref{Bresse Equation 2.35}, we get:
	\begin{eqnarray} 
\rho_1 \lambda^2 \widehat v^1+\ell  k_3\left(\widehat v^5_x-\ell  \widehat v^1\right)&=&0 \quad \mathrm{in} \quad (\alpha, \beta),\label {Bresse Equation 2.36}\\ \noalign{\medskip}
\widehat v^1_x+\ell  \widehat v^5&=&0   \quad \mathrm{in} \quad (\alpha, \beta),\label {Bresse Equation 2.37}\\ \noalign{\medskip}
\rho_1 \lambda^2 \widehat v^5+k_3\left(\widehat v^5_x-\ell  \widehat v^1\right)_x&=&0 \quad \mathrm{in} \quad (\alpha, \beta). \label {Bresse Equation 2.38}
	\end{eqnarray}
	Differentiating equation \eqref{Bresse Equation 2.36} with respect to $x$, a straightforward computation with equation  \eqref{Bresse Equation 2.38} yields:
	\begin{equation*}
\rho_1	\lambda ^2 \left(\widehat v_x^1 - \ell  \widehat v^5\right) = 0 \quad \mathrm{in} \quad (\alpha, \beta).
	\end{equation*}  
	Equivalently
	\begin{equation} \label{Bresse Equation 2.39} 
\widehat v_x^1 - \ell  \widehat v^5 = 0 \quad \mathrm{in} \quad (\alpha, \beta).
	\end{equation} 
	Hence, from equations \eqref{Bresse Equation 2.37} and \eqref{Bresse Equation 2.39}, we get:
	\begin{equation}\label{Bresse strong.2}
\widehat v^5 =0 \quad \mathrm{and} \quad \widehat v^1_x =0 \quad  \mathrm{in}  \quad (\alpha,\beta). 
	\end{equation}
	{Plugging} $\widehat v^5=0 $ in  \eqref{Bresse Equation 2.36}, we get:
	\begin{equation}\label{Bresse Equation 2.40} 
\left(\rho_1 \lambda^2- \ell ^2 k_3\right) \widehat v^1=0.
	\end{equation}
	{In order to finish} our proof, {we have to distinguish two cases:} \\
	\underline{\bf Case 1:} $\lambda \neq \ell \ \sqrt{\dfrac{k_3}{\rho_1}}$.\\
	%%%%%%%%%%%%%%%%%%%%%%%%%%%%%%%%%%%%%%%%%%%%%%%%%%%%%%%%%%%
	%$\lambda \neq \ell \sqrt{\frac{k_3}{\rho_1}}$%
	%%%%%%%%%%%%%%%%%%%%%%%%%%%%%%%%%%%%%%%%%%%%%%%%%%%%%%%%%%%
	Using equation \eqref{Bresse Equation 2.40} , we deduce that:
	$$\widehat v^1=0 \quad \mathrm{in} \quad (\alpha, \beta).$$
	Setting $V= \left(\widehat v^1,\widehat v^1_x,\widehat v^3,\widehat v^3_x,\widehat v^5,\widehat v^5_x \right)^\mathsf{T}$. By continuity of $\widehat v^i$ on $(0,L)$,
we deduce that $V(\alpha)=0$.
	Then system \eqref{Bresse strong} could be given as:
	\begin{equation} \label{SystemV.2}
\left\{
\begin{matrix}
V_x &= &B V, \ \ \ {\rm{in}} \ \ (0, \alpha)\\
V(\alpha) &=&0,
\end{matrix}
\right.
	\end{equation}
	where 
	\begin{equation}
B=
\begin{pmatrix}
0 & 1 & 0 & 0 & 0 & 0\\
\dfrac{-\lambda^2 \rho_1+ \ell ^2k_3}{k_1} & 0 & 0 & -1 & \dfrac{-\ell (k_1+k_3)}{k_1}&0\\
0 & 0 & 0& 1 & 0 & 0\\
0 & \dfrac{k_1}{k_2+i \lambda D_2} & \dfrac{k_1- \lambda^2 \rho_2}{k_2+i \lambda D_2} & 0 & \dfrac{\ell k_1}{k_2+i \lambda D_2}&0\\
0 & 0 & 0 & 0 & 0 & 1\\
0 & \dfrac{\ell (k_3 +k_1)}{k_3}& \dfrac{\ell k_1 }{k_3}& 0 & \dfrac{\ell ^2k_1 - \lambda^2\rho_1}{k_3}& 0
\end{pmatrix}.
\end{equation}
Using ordinary differential equation theory, we deduce that system \eqref{SystemV.2} has the unique trivial solution $V=0$ in $(0,\alpha)$.
The same argument as above leads us to prove that $V=0$ on $(\beta, L)$. 
Consequently, we obtain $\widehat v^1 =  \widehat v^3 = \widehat v^5 = 0$ on $(0,L)$. It follows that $\widehat v^2 =  \widehat v^4 = \widehat v^6 = 0$
on $(0,L)$, thus $\widehat U =0$. This gives that $U=C$, where $C$ is a constant. Finally, from the boundary condition \eqref{DDDD} or \eqref{DNND}, we deduce that $U=0$.\\ 
%%%%%%%%%%%%%%%%%%%%%%%%%%%%%%%%%%%%%%%%%%%%%%%%%%%%%%%%%%%
%$\lambda =\ell \sqrt{\frac{k_3}{\rho_1}}$%
%%%%%%%%%%%%%%%%%%%%%%%%%%%%%%%%%%%%%%%%%%%%%%%%%%%%%%%%%%%
\underline{\bf Case 2:} $\lambda=\ell \ \sqrt{\dfrac{k_3}{\rho_1}}$.\\
The fact that $ \widehat v^1_x=0$ on $(\alpha, \beta)$, we get $\widehat v^1 = c$ on $(\alpha, \beta)$, where $c$ is a constant.
By continuity of $\widehat v^1$ on $(0,L)$, we deduce that $\widehat v^1(\alpha)=c$.
We know also that $\widehat v^3 = \widehat v^5=0$ on $(\alpha, \beta)$ from \eqref{Bresse strong.1} and \eqref{Bresse strong.2}.
Hence, setting $V(\alpha)= (c,0,0,0,0,0)^\mathsf{T} = V_0$, we can rewrite system \eqref{Bresse strong} on  $(0,\alpha)$ under the form: 
\begin{equation*} 
\left\{
\begin{matrix}
V_x &= &\widehat B V, \\
V(\alpha) &=&V_0,
\end{matrix}
\right.
	\end{equation*}
	where 
	\begin{equation*}
\widehat B=
\begin{pmatrix}
0 & 1 & 0 & 0 & 0 & 0\\
0 & 0 & 0 & -1 & \dfrac{-\ell (k_1+k_3)}{k_1}&0\\
0 & 0 & 0& 1 & 0 & 0\\
0 & \dfrac{k_1}{k_2+i \ell \sqrt{\frac{k_3}{\rho_1}} D_2} & \dfrac{k_1- \lambda^2 \rho_2}{k_2+i \ell \sqrt{\frac{k_3}{\rho_1}}\lambda D_2} & 0 & \dfrac{\ell k_1}{k_2+i \ell \sqrt{\frac{k_3}{\rho_1}} D_2}&0\\
0 & 0 & 0 & 0 & 0 & 1\\
0 & \dfrac{\ell (k_3 +k_1)}{k_3}& \dfrac{\ell k_1 }{k_3}& 0 & \dfrac{\ell ^2(k_1 - k_3)}{k_3}& 0
\end{pmatrix}.
\end{equation*}
Introducing $\widetilde V= \left(\widehat v^1_x,\widehat v^3,\widehat v^3_x,\widehat v^5,\widehat v^5_x \right)^\mathsf{T}$ and
\begin{equation*}
\widetilde B=
\begin{pmatrix}
0 & 0 & -1 & \dfrac{-\ell (k_1+k_3)}{k_1}&0\\
0 & 0& 1 & 0 & 0\\
\dfrac{k_1}{k_2+i \ell \sqrt{\frac{k_3}{\rho_1}} D_2} & \dfrac{k_1- \lambda^2 \rho_2}{k_2+i \ell \sqrt{\frac{k_3}{\rho_1}}\lambda D_2} & 0 & \dfrac{\ell k_1}{k_2+i \ell \sqrt{\frac{k_3}{\rho_1}} D_2}&0\\
0 & 0 & 0 & 0 & 1\\
\dfrac{\ell (k_3 +k_1)}{k_3}& \dfrac{\ell k_1 }{k_3}& 0 & \dfrac{\ell ^2(k_1 - k_3)}{k_3}& 0
\end{pmatrix}.
\end{equation*}
Then system \eqref{Bresse strong} could be given as:
\begin{equation} \label{Sam2}
\left\{
\begin{matrix}
\widetilde	V_x &= &\widetilde B\widetilde V, \ \ \mathrm{in}  \ \ (0, \alpha), \\
\widetilde	V(\alpha) &=&0.
\end{matrix}
\right.
\end{equation}
Using ordinary differential equation theory, we deduce that system \eqref{Sam2} has the unique trivial solution  $\widetilde V=0$ in $(0, \alpha)$.
This implies that on $(0, \alpha)$, we have $\widehat v^3 = \widehat v^5=0$. Consequently, $v^3=c_3$  and $v^5=c_5$ where $c_3$ and $c_5$ are constants.
But using the fact that $v^3(0)=v^5(0)=0$, we deduce that $v^3=v^5=0$ on $(0, \alpha)$.\\
Substituting $v^3$ and $v^5$ by their values in the second equation of system \eqref{Bresse strong}, we get that $v^1_x=0$. This yields $v^1=c_1$, where $c_1$ is a constant.
But as $v^1(0)=0$, we get: $v^1=0$ on $(0,\alpha)$.
Thus $U=0$ on $(0, \alpha)$. 
The same argument as above leads us to prove that $U=0$ on $(\beta, L)$ and therefore $U=0$ on $(0,L)$. Thus the proof is complete. 
\end{proof}
%%%%%%%%%%%%%%%%%%%%%%%%%%%%%%%%%%%%%%%%%%%%%%%%%%
% Lemma 2.7%
%%%%%%%%%%%%%%%%%%%%%%%%%%%%%%%%%%%%%%%%%%%%%%%%%%
\begin{lem}\label{Bresse Lemma 2.7}
	Under the same condition of Theorem \ref{Bresse Theorem 2.5}, $(\ i\lambda I  -\mathcal{A}_j), j=1, 2$ is surjective for all $\lambda\in\mathbb{R}$.
\end{lem}
%%%%%%%%%%%%%%%%%%%%%%%%%%%%%%%%%%%%%%%%%%%%%%%%%%
% Proof of Lemma 2.7%
%%%%%%%%%%%%%%%%%%%%%%%%%%%%%%%%%%%%%%%%%%%%%%%%%%
\begin{proof} We will prove Lemma \ref{Bresse Lemma 2.7} in the case $D_1=D_3=0$ on $(0,L)$ and $D_2 \geq d_0 >0$ on $(\alpha, \beta) \subset (0,L)$
and the other cases are similar to prove.\\
	Since $0\in\rho\left(\mathcal{A}_j\right)$,  we  still need to show the result for $\lambda\in\mathbb{R^*}$.
For any  $$F=\left(f^1,f^2,f^3,f^4,f^5,f^6\right)^{\mathsf{T}}\in\mathcal{H}_j,\ \lambda\in \mathbb{R}^*,$$ 
we prove the existence of $$U=\left(v^1,v^2,v^3,v^4,v^5,v^6\right)^{\mathsf{T}}\in D\left(\mathcal{A}_j\right)$$ solution of the following equation:
	\begin{equation}\label{surj}
\left(\ i\lambda I -\mathcal{A}_j\right)U=F.
	\end{equation}
	%%%%%%%%%%%%%%%%%%%%%%%%%%%%%%%%%%%%%%%%%%%%%%%%%%
	Equivalently, we have the following system:
	\begin{eqnarray}
\ i\lambda v^1-v^2&=&f^1, \label{Bresse Equation 2.41}
\\
\ \rho_1\ i\lambda v^2-k_1 \left(v^1_x+v^3+\ell  v^5\right)_x-\ell k_3\left(v^5_x-\ell  v^1\right)&=&\rho_1f^2,\label{Bresse Equation 2.42}
\\
\ i\lambda v^3-v^4&=&f^3,\label{Bresse Equation 2.43}
\\
\rho_2\ i\lambda v^4-(k_2v^3_x+D_2v^4_x)_x+ k_1\left(v^1_x+ v^3 +\ell v^5\right)
&=&\rho_2 f^4,\label{Bresse Equation 2.45}
\\
\	i \lambda v^5 -v^6 &=& f^5,\label{Bresse Equation 2.46}
\\
\rho_1\ i\lambda v^6-k_3\left(v^5_x-\ell  v^1\right)_x+\ell  k_1\left(v^1_x+v^3+\ell  v^5\right)&=&\rho_1 f^6 .\label{Bresse Equation 2.47}
	\end{eqnarray}
	%%%%%%%%%%%%%%%%%%%%%%%%%%%%%%%%%%%%%%%%%%%%%%%%%%
	%%%%%%%%%%%%%%%%%%%%%%%%%%%%%%%%%%%%%%%%%%%%%%%%%%
	From \eqref{Bresse Equation 2.41},\eqref{Bresse Equation 2.43} and \eqref{Bresse Equation 2.46}, we have:
	\begin{equation}\label{Bresse Equation 2.49}
v^2=\ i\lambda v^1-f^1,\quad v^3=\ i\lambda v^3-f^3,\quad v^6=\ i\lambda v^5-f^5.
	\end{equation}
	%%%%%%%%%%%%%%%%%%%%%%%%%%%%%%%%%%%%%%%%%%%%%%%
	%%%%%%%%%%%%%%%%%%%%%%%%%%%%%%%%%%%%%%%%%%%%%%%
	%%%%%%%%%%%%%%%%%%%%%%%%%%%%%%%%%%%%%%%%%%%%%%%%%%
	Inserting \eqref{Bresse Equation 2.49} in \eqref{Bresse Equation 2.42}, \eqref{Bresse Equation 2.45} and \eqref{Bresse Equation 2.47}, we get:
	\begin{equation}\label{Bresse Equation 2.44}
\left\{
\begin{array}{ll}
\displaystyle{-\lambda^2 v^1-k_1 \rho_1^{-1}\left(v^1_x+v^3+\ell  v^5\right)_x-\ell  k_3\rho_1^{-1}\left(v^5_x-\ell v^1\right)
	=h^1,}
\\ \noalign{\medskip}
\displaystyle{-\lambda^2 v^3-\rho_2^{-1}\left(k_2 +i \lambda D_2\right) v^3_{xx}+k_1\rho_2^{-1}\left(v^1_x+v^3+\ell  v^5\right)=h^3 ,}
\\ \noalign{\medskip}
\displaystyle{-\lambda^2 v^5-k_3\rho_1^{-1}\left(v^5_x-\ell  v^1\right)_x+\ell  k_1\rho_1^{-1}\left(v^1_x+v^3+\ell  v^5\right)=h^5 ,}
\end{array}
\right.
	\end{equation}
	where
	\begin{equation*}
\displaystyle{h^1=f^2+\ i\lambda f^1,\quad h^3=f^4+\ i\lambda f^3- \rho_2^{-1}D_2f^3_{xx},\quad h^5=f^6+i \lambda f^5.}
	\end{equation*}
	%%%%%%%%%%%%%%%%%%%%%%%%%%%%%%%%%%%%%%%%%%%%%%%%%%
	For all $v=\left(v^1,v^3,v^5\right)^{\mathsf{T}}\in \left(H^1_0(0,L)\right)^3$ for $j=1$ and $v=\left(v^1,v^3,v^5\right)^{\mathsf{T}}\in H^1_0(0,L)\times H_*^1\left(0,L\right)^2$
for $j=2$, we define the linear operator $\mathcal{L}$ by:
	\begin{equation*}
\mathcal{L}v=\left(\begin{array}{c}
\displaystyle{-k_1 \rho_1^{-1}\left(v^1_x+v^3+\ell  v^5\right)_x-\ell  k_3 \rho_1^{-1}\left(v^5_x-\ell v^1\right)
}
\\  \noalign{\medskip}
\displaystyle{- \rho_2^{-1}(k_2+i \lambda D_2)v^3_{xx}+k_1\rho_2^{-1}\left(v^1_x+v^3+\ell  v^5\right) }
\\  \noalign{\medskip}
\displaystyle{-k_3\rho_1^{-1}\left(v^5_x-\ell  v^1\right)_x+\ell  k_1\rho_1^{-1}\left(v^1_x+v^3+\ell  v^5\right)}
\end{array}\right).
	\end{equation*}
	%%%%%%%%%%%%%%%%%%%%%%%%%%%%%%%%%%%%%%%%%%%%%%%%%%%%
	For clarity, we consider the case $j=1$. The proof in the case $j=2$ {is} very similar. Using Lax-Milgram theorem, it is easy to show that $\mathcal{L}$ is an
isomorphism from $(H^1_0(0,L))^3$ onto $(H^{-1}\left(0,L\right))^3$. Let $v=\left(v^1,v^3,v^5\right)^{\mathsf{T}}$ and  $h=\left(-h^1,-h^3,-h^5\right)^{\mathsf{T}}$, then we transform system \eqref{Bresse Equation 2.44} into the following form:
	\begin{equation}\label{Bresse Equation 2.50}
(\lambda^2\mathcal{I}-\mathcal{L})v=h.
	\end{equation}
	%%%%%%%%%%%%%%%%%%%%%%%%%%%%%%%%%%%%%%%%%%%%%%%%%%
	Since the operator $\mathcal{L}$ is an isomorphism from $(H^1_0(0,L))^3$ onto $(H^{-1}\left(0,L\right))^3$ and $\mathcal{I}$ is a compact operator from $(H^1_0(0,L))^3$
onto $(H^{-1}\left(0,L\right))^3$, then using Fredholm's Alternative theorem, problem \eqref{Bresse Equation 2.50} admits a unique solution in 
	$(H^1_0(0,L))^3$ if and only if $\lambda^2\mathcal{I}-\mathcal{L}$ is injective. For that purpose, let $\tilde{v}=\left(\tilde{v}^1,\tilde{v}^3,\tilde{v}^5\right)^{\mathsf{T}}$
in ${\rm ker}(\lambda^2\mathcal{I}-\mathcal{L})$. Then, if we set $\tilde{v}^2=i\lambda \tilde{v}^1$, $\tilde{v}^4=i\lambda \tilde{v}^3$ and $\tilde{v}^6=i\lambda \tilde{v}^5$,
we deduce that $\tilde{V}=(\tilde{v}^1, \tilde{v}^2, \tilde{v}^3, \tilde{v}^5, \tilde{v}^6)$  belongs to  $D(\AA_1)$ and it is solution of:
	$$(i\lambda\mathcal{I}-\AA_1)\tilde{V}=0.$$
	Using Lemma \ref{Bresse Lemma 2.6}, we deduce that $\tilde{v}^1=\tilde{v}^3=\tilde{v}^5=0$. This implies that equation \eqref{Bresse Equation 2.50}
admits a unique solution in $v=(v^1, v^3, v^5)\in (H^1_0(0,L))^3$ and 
	$$ \displaystyle{-k_1 \rho_1^{-1}\left(v^1_x+v^3+\ell  v^5\right)_x-\ell  k_3 \rho_1^{-1}\left(v^5_x-\ell v^1\right)}\in L^2(0,L),$$
	$$\displaystyle{- \rho_2^{-1}(k_2v^3_{x}+i \lambda D_2v^3_{x}-D_2f^3_{x})_x+k_1\rho_2^{-1}\left(v^1_x+v^3+\ell  v^5\right) }\in L^2(0,L),$$ 
	$$\displaystyle{-k_3\rho_1^{-1}\left(v^5_x-\ell  v^1\right)_x+\ell  k_1\rho_1^{-1}\left(v^1_x+v^3+\ell  v^5\right)}\in L^2(0,L).$$
	By setting $v^2=\ i\lambda v^1-f^1$, $v^3=\ i\lambda v^3-f^3$ and  $v^6=\ i\lambda v^5-f^5$, we deduce that {$V=(v^1,v^2,v^3,v^4,v^5,v^6)$} belongs to $D(\AA_1)$
	and it is the unique solution of equation \eqref{surj} and the proof is thus complete. 
	\end{proof}
%%%%%%%%%%%%%%%%%%%%%%%%%%%%%%%%%%%%%%%%%%%%%%%%%%
% Proof of Theorem 2.8%
%%%%%%%%%%%%%%%%%%%%%%%%%%%%%%%%%%%%%%%%%%%%%%%%%%
\begin{proof}[Proof of Theorem \ref{Bresse Theorem 2.5}.]Following a general criteria of Arendt-Batty in \cite{Arendt-Batty88}, the $C_0-$semigroup $e^{t\mathcal{A}_j}$
of contractions is strongly  stable if $\mathcal{A}_j$ has no pure imaginary eigenvalues and $\sigma(\mathcal{A}_j)\cap i\mathbb{R}$ is countable.
By Lemma \ref{Bresse Lemma 2.6}, the operator 
$\mathcal{A}_j$ has no pure imaginary eigenvalues and by Lemma \ref{Bresse Lemma 2.7}, ${\rm R}(\ i\lambda -\mathcal{A}_j)=\mathcal{H}_j$  for all $\lambda\in\mathbb{R}$.
Therefore the closed graph theorem of Banach  implies that $\sigma(\mathcal{A}_j)\cap i\mathbb{R}=\emptyset$. Thus, the proof is complete.
\end{proof}
%%%%%%%%%%%%%%%%%%%%%%%%%%%%%%%%%%%%%%%%%%%%%%%%%%%%%%%%%%%%%%%%%%%%%%%%%%%%%%%%%%%%%%%%%%%%%%%%%%%%% 
% POLYNOMIAL STABILITY IN THE CASE OF 3 LOCAL DAMPINGS%
%%%%%%%%%%%%%%%%%%%%%%%%%%%%%%%%%%%%%%%%%%%%%%%%%%%%%%%%%%%%%%%%%%%%%%%%%%%%%%%%%%%%%%%%%%%%%%%%%%%%%%
\section{Polynomial stability for non smooth damping coefficients at the interface}\label{Polynomial1}
Before we state our main result, we recall the following results (see \cite{Huang-85}, \cite{Pruss-84} for part \textit{i)}, \cite{Borichev-Tomilov10}
for \textit{ii)} and \cite{Pazy83} for \textit{iii)}.
\begin{thm}\label{type} 
	Let $\AA: D(\AA) \subset \HH \rightarrow \HH$ be an unbounded operator generating a $C_0$-semigroup of contractions $e^{t\AA}$ on $\HH$. Assume that $i\lambda \in \rho(\AA)$, for all $\lambda \in \mathbb{R}$. Then, the $C_0$-semigroup $e^{t\AA}$ is: \\
	i) Exponentially stable if and only if 
	$$\lim_{|\lambda| \rightarrow + \infty}\bigg\{\sup_{\lambda \in \mathbb{R}} \norm{\left(i \lambda I -\AA\right)^{-1}}_{\mathcal{L}{(\HH)}}\bigg\} < +\infty.$$	
	ii) Polynomially stable of order $\frac{1}{l}$ $(l>0)$	if and only if
	$$\lim_{|\lambda| \rightarrow + \infty}\bigg\{\sup_{\lambda \in \mathbb{R}} |\lambda|^{-l}\norm{\left(i \lambda I -\AA\right)^{-1}}_{\mathcal{L}{(\HH)}} \bigg\}< +\infty.$$	
	iii) Analytically stable if and only if
	$$\lim_{|\lambda| \rightarrow + \infty}\bigg\{\sup_{\lambda \in \mathbb{R}} |\lambda| \norm{\left(i \lambda I -\AA\right)^{-1}}_{\mathcal{L}{(\HH)}} \bigg\}< +\infty.$$
\end{thm}

It was proved that, see \cite{ChenLiuLiu-1998,Liu-Liu-2002}, the stabilization of wave equation with local Kelvin-Voigt damping is greatly influenced by
the smoothness of the damping coefficient and the region where the damping is localized (near or faraway from the boundary) even in the one-dimensional case.
So, in this section, we consider the Bresse systems \eqref{eqq1.1'}-\eqref{DDDD} and  \eqref{eqq1.1'}-\eqref{DNND} subject to three local viscoelastic Kelvin-Voigt
dampings with non smooth coefficients at the interface. Using frequency domain approach combined with multiplier techniques and the construction of {a} new multiplier function,
we establish {the} polynomial stability of the $C_0$-semigroup $e^{t\AA_j}$, $j=1,2$. For this purpose, let $\emptyset \not=(\alpha_i,\beta_i)\subset (0,L)$, $i=1,2,3,$
be {an} arbitrary nonempty open subsets of $(0,L)$. We consider the following stability condition:
\begin{equation} \label{3local.1}
\exists\, \, d_0^i >0 \,\,\, \mbox{such that} \,\, D_i\geq d_0^i >0 \, \, {\rm in} \, \,   (\alpha_i,\beta_i), \, \, i=1,2,3, \, \,
{\rm and} \, \, \bigcap_{i=1}^3(\alpha_i,\beta_i)=(\alpha,\beta)\not=\emptyset.
\end{equation}
%%%%%%%%%%%%%%%%%%%%%%%%%%%%%%%%%%%%%%%%%%%%%%%%%%%
Our main result in this section can be given by the following theorem:
\begin{thm}\label{PolyThe.2}
Assume that condition \eqref{3local.1} holds. 
Then, there exists a positive constant $c>0$ such that for all $U_0 \in D{(\AA_j)}$, $j=1,2,$  the energy of the system satisfies the following decay rate:
\begin{equation}\label{polydecay.2}
E(t)\leq  \frac{c}{t}\norm{U_0}^2_{D{(\AA_j)}}.
\end{equation}
\end{thm}
%%%%%%%%%%%%%%%%%%%%%%%%%%%%%%%%%%%%%%%%%%%%%%%%%%%
%PROOF
%%%%%%%%%%%%%%%%%%%%%%%%%%%%%%%%%%%%%%%%%%%%%%%%%%%
Referring to \cite{Borichev-Tomilov10}, \eqref{polydecay.2} is verified if the following conditions
\begin{equation*}\tag {\rm H1}
\ i\mathbb{R}\subseteq\rho\left(\mathcal{A}_j\right) 
\end{equation*}
and
\begin{equation*}\tag {\rm H3}
\lim_{\lambda \rightarrow +\infty} \sup_{\lambda\in\mathbb{R}}\bigg\{\frac{1}{\lambda^2}\left\|\left(i\lambda Id-\mathcal{A}_j\right)^{-1}\right\|_{\mathcal{L}\left(\mathcal{H}_j\right)}\bigg\}=O\left(1\right)
\end{equation*}
hold. \\
%%%%%%%%%%%%%%%%%%%%%%%%%%%%%%%%%%%%%%%%%%%%%%%%%%
Condition  $\ i\mathbb{R}\subseteq\rho\left(\mathcal{A}_j\right)$ is  already proved in Lemma \ref{Bresse Lemma 2.6} and Lemma \ref{Bresse Lemma 2.7}. 

We will establish (H3) by  contradiction. Suppose that there exist a sequence of real numbers $\left(\lambda_n\right)_n$, with  $|\lambda_n|\to+\infty$  and a sequence of vectors
\begin{equation}\label{Bresse equation 2.1}
	U_n=\left(v^1_n,v^2_n,v_n^3,v^4_n,v^5_n,v^6_n\right)^{\mathsf{T}}\in D\left(\mathcal{A}_j\right) \ \text{ with }\ \|U_n\|_{\mathcal{H}_j}=1
\end{equation}
%%%%%%%%%%%%%%%%%%%%%%%%%%%%%%%%%%%%%%%%%%%%%%%%%%
such that
\begin{equation}\label{Bresse equation 2.2}
	\lambda_n^2\left(\ i\lambda_nU_n-\mathcal{A}_jU_n\right)=\left(f^1_n,f^2_n,f_n^3,f^4_n,f^5_n,f^6_n\right)^{\mathsf{T}}\to 0\ \text{ in } \mathcal{H}_j,\quad j=1,2.
\end{equation}
%%%%%%%%%%%%%%%%%%%%%%%%%%%%%%%%%%%%%%%%%%%%%%%%%%
%%%%%%%%%%%%%%%%%%%%%%%%%%%%%%%%%%%%%%%%%%%%%%%%%%
We will check the condition (H3) by finding a contradiction with \eqref{Bresse equation 2.1}-\eqref{Bresse equation 2.2} such as $\left\| U_n\right\|_{\mathcal{H}_j} =o(1)$.

Equation \eqref{Bresse equation 2.2} is detailed as:
\begin{eqnarray}
	\ i\lambda_n v^1_n-v^2_n&=&\frac{f^1_n}{\lambda_n^2},\label{Bresse equation 2.3}
	\\
	i\rho_1\lambda_n v^2_n-\left[k_1 \left(\left(v^1_n\right)_x+v^3_n+\ell  v^5_n\right) +D_1 \left( \left(v^2_n\right)_x+v_n^4+\ell v^6_n\right)\right]_x\nonumber \\-  \ell  k_3\left[\left(v^5_n\right)_x-\ell  v^1_n\right] - \ell  D_3 \left[\left(v_n^6\right)_x - \ell  v_n^2\right]& =& \rho_1\frac{f^2_n}{\lambda_n^2},\label{Bresse equation 2.4}
	\\
	\ i\lambda_n v^3_n-v^4_n&=&\frac{f^3_n}{\lambda_n^2},\label{Bresse equation 2.5}
	\\
	\ i \rho_2 \lambda_n v^4_n -\left[k_2 \left(v^3_n\right)_x + D_2 \left(v^4_n\right)_x\right]_{x}+ k_1\left[\left(v^1_n\right)_x+v^3_n+\ell  v^5_n\right]+D_1 \left[\left(v^2_n\right)_x +v^4_n+ \ell  v^6_n\right]
	&=& \rho_2\frac{f^4_n}{\lambda_n^2},\label{Bresse equation 2.6}
	\\
	i\lambda_n v^5_n-v^6_n&=& \frac{f^5_n}{\lambda_n^2},\label{Bresse equation 2.7}\\
	\ i \rho_1 \lambda_n v^6_n-\left[k_3\left(\left(v^5_n\right)_x-\ell  v^1_n\right) + D_3 \left(\left(v^6_n\right)_x - \ell  v^2_n\right)\right]_x + \ell  k_1\left[\left(v^1_n\right)_x+v^3_n+\ell  v^5_n\right] \nonumber\\
	+ \ell  D_1 \left[\left(v^2_n\right)_x+ v^4_n +\ell v_n^6\right]&=&\rho_1 \frac{f^6_n}{\lambda_n^2}.\label{Bresse equation 2.8}
\end{eqnarray}
%%%%%%%%%%%%%%%%%%%%%%%%%%%%%%%%%%%%%%%%%%%%%%%%%%%
From \eqref{Bresse equation 2.1}, \eqref{Bresse equation 2.3}, \eqref{Bresse equation 2.5} and \eqref{Bresse equation 2.7}, we deduce that:
\begin{equation}\label{Bresse equation 2.9}
	\norm{v^1_n}=O(\frac{1}{\lambda_n}), \ \  \norm{v^3_n}=O(\frac{1}{\lambda_n}), \ \
	\norm{v^5_n}=O(\frac{1}{\lambda_n}).
\end{equation}
%%%%%%%%%%%%%%%%%%%%%%%%%%%%%%%%%%%%%%%%%%%%%%%%%%%
For clarity, we divide the proof into several lemmas. From  now on, for simplicity, we drop the index $n$. 
%%%%%%%%%%%%%%%%%%%%%%%%%%%%%%%%%%%%%%%%%%%%%%%%%%%
%LEMMA
%%%%%%%%%%%%%%%%%%%%%%%%%%%%%%%%%%%%%%%%%%%%%%%%%%%
%%%%%%%%%%%%%%%%%%%%%%%%%%%%%%%%%%%%%%%%%%%%%%%%%%	
\begin{lem}\label{Bresse polydiss.1}
	Under all the above assumptions, we have:
	\begin{equation}\label{Bresse equation 2.10} 
\norm{D_1^{1/2}\left(v^2_{x}+v^4 +\ell v^6\right)}=\frac{o\left(1\right)}{\lambda}, \ \ \norm{D_2^{1/2}v^4_{x}}=\frac{o\left(1\right)}{\lambda},\ \ \norm{D_3^{1/2}\left(v^6_{x}-\ell v^2\right)}=\frac{o\left(1\right)}{\lambda} 
%\quad \mathrm{in} \quad L^2\left(0,L\right)
	\end{equation}
	and 
	\begin{equation} \label{Bresse equation 2.11}
\norm{v^2_{x}+v^4 +\ell v^6}=\frac{o\left(1\right)}{\lambda}, \ \ \norm{v^4_{x}}=\frac{o\left(1\right)}{\lambda},\ \ \norm{v^6_{x}-\ell v^2}=\frac{o\left(1\right)}{\lambda} \quad \mathrm{in} \quad \left(\alpha,\beta\right).
	\end{equation}
\end{lem}
%%%%%%%%%%%%%%%%%%%%%%%%%%%%%%%%%%%%%%%%%%%%%%%%%%%
%PROOF
%%%%%%%%%%%%%%%%%%%%%%%%%%%%%%%%%%%%%%%%%%%%%%%%%%%
\begin{proof}Taking the inner product of \eqref{Bresse equation 2.2} with $U$ in $\mathcal{H}_j$, we get:
\begin{align}\label{Bresse equation 2.12}
\mathrm{Re}\left(\ i\lambda^3 \norm{U}^2-\lambda^2\left(\mathcal{A}_jU,U\right)\right)_{\mathcal{H}_j}&=-\lambda^2\mathrm{Re}\left(\mathcal{A}_jU,U\right)_{\mathcal{H}_j} \nonumber\\
&
=\lambda^2\int_0^L \left(D_1|v^2_{x}+v^4 +\ell v^6|^2 +D_2|v^4_{x}|^2 + D_3|v^6_{x}-\ell v^2|^2\right)dx= o\left(1\right).
\end{align}
%%%%%%%%%%%%%%%%%%%%%%%%%%%%%%%%%%%%%%%%%%%%%%%%%%%
Thanks to \eqref{3local.1}, we obtain the desired asymptotic equations \eqref{Bresse equation 2.10} and \eqref{Bresse equation 2.11}. Thus the proof is complete.
\end{proof}
{
\begin{rk}
These estimates are crucial for the rest of the prooof and they will be used to prove each point of the global proof  divided in several lemmas.
\end{rk}
}%%%%%%%%%%%%%%%%%%%%%%%%%%%%%%%%%%%%%%%%%%%%%%%%%%%
%%%%%%%%%%%%%%%%%%%%%%%%%%%%%%%%%%%%%%%%%%%%%%%%%%%
%LEMMA
%%%%%%%%%%%%%%%%%%%%%%%%%%%%%%%%%%%%%%%%%%%%%%%%%%%%%%%%%%%%%%%%%%%%%%%%%%%%%%%%%%%%%%%%%%%%%%%%%%%%%%
\begin{lem}\label{Bresse polydiss2}
Under all the above assumptions, we have:
\begin{equation} \label{Bresse equation 2.13}
\norm{v^1_{x}+v^3 +\ell v^5}=\frac{o\left(1\right)}{\lambda^2}, \ \ \norm{v^3_{x}}=\frac{o\left(1\right)}{\lambda^2},\ \ \norm{v^5_{x}-\ell v^1}=\frac{o\left(1\right)}{\lambda^2} \quad \mathrm{in} \quad \left(\alpha,\beta\right).
\end{equation}
\end{lem}
%%%%%%%%%%%%%%%%%%%%%%%%%%%%%%%%%%%%%%%%%%%%%%%%%%%%%%%%%
%PROOF
%%%%%%%%%%%%%%%%%%%%%%%%%%%%%%%%%%%%%%%%%%%%%%%%%%%%%%%%%%

\begin{proof} First, using equations \eqref{Bresse equation 2.3}, \eqref{Bresse equation 2.5} and \eqref{Bresse equation 2.7}, we obtain:
\begin{equation}
\lambda \left(v^1_x +v^3 +\ell  v^5\right)= -i(v_x^2 +\frac{f_x^1}{\lambda^2}+v^4+\frac{f^3}{\lambda^2}+\ell v^6 +\ell \frac{f^5}{\lambda^2}).
\end{equation}
Consequently, 
\begin{equation}\label{Bresse Equation 12.2}
\int_{\alpha}^{\beta}\lambda^2 |v^1_x +v^3 +\ell  v^5|^2dx \leq 2\int_{\alpha}^{\beta}|v_x^2 +v^4+\ell v^6|^2dx + 2\int_{\alpha}^{\beta}\frac{|f_x^1 +f^3+\ell f^5|^2}{\lambda^4}dx.
\end{equation}
Using the first  {estimate} of \eqref{Bresse equation 2.11} and the fact that $f^1$, $f^3$, $f^5$ converge to zero in $H_0^1(0,L)$ (or in $H_\star^1(0,L)$) in \eqref{Bresse Equation 12.2}, 
we deduce:
\begin{equation}
\int_{\alpha}^{\beta}\lambda^2 |v^1_x +v^3 +\ell  v^5|^2dx = \frac{o(1)}{\lambda^2}.
\end{equation}
In a similar way, one can prove:
\begin{equation}
\int_{\alpha}^{\beta}\lambda^2 |v^3_x|^2dx=\frac{o(1)}{\lambda^2} \quad \mathrm{and} \quad
\int_{\alpha}^{\beta}\lambda^2 |v^5_x -\ell  v^1|^2dx = \frac{o(1)}{\lambda^2}.
\end{equation}
The proof is thus complete.
\end{proof}
%%%%%%%%%%%%%%%%%%%%%%%%%%%%%%%%%%%%%%%%%%%
Here and after $\epsilon$ designates a fixed positive real number such that $0 <\alpha+\epsilon <\beta-\epsilon<  L$. 
Then, we define the cut-off function $\eta \in C_c^{\infty}{\left(\mathbb{R}\right)}$ by: 
\[
\eta =1  \ \ \mathrm{on} \ \ \left[\alpha+\epsilon, \beta-\epsilon\right], \ \ 0 \leq \eta \leq 1, \ \ \eta=0 \ \ \mathrm{on} \ \ \left(0,L\right)\setminus \left(\alpha,\beta\right). 
\]

%%%%%%%%%%%%%%%%%%%%%%%%%%%%%%%%%%%%%%%%%%%%%%%%%%%
%LEMMA
%%%%%%%%%%%%%%%%%%%%%%%%%%%%%%%%%%%%%%%%%%%%%%%%%%%%%%%%%%%%%%%%%%%%%%%%%%%%%%%%%%%%%%%%%%%%%%%%%%%%%%
\begin{lem}
Under all the above assumptions, we have:
\begin{equation}\label{Bresse equation 2.14}
\int_{\alpha+\epsilon}^{\beta-\epsilon}\left|\lambda v^1\right|^2dx=\frac{o\left(1\right)}{\lambda}, \ \
\int_{\alpha+\epsilon}^{\beta-\epsilon}\left|\lambda v^3\right|^2dx=\frac{o\left(1\right)}{\lambda}, \ \
\int_{\alpha+\epsilon}^{\beta-\epsilon}\left|\lambda v^5\right|^2dx=\frac{o\left(1\right)}{\lambda}.
\end{equation}
\end{lem}
%%%%%%%%%%%%%%%%%%%%%%%%%%%%%%%%%%%%%%%%%%%%%%%%%%%%%%
\begin{proof}
	First, multiplying equation \eqref{Bresse equation 2.3} by $i \lambda \eta \overline{v^1}$ in $L^2(0,L)$ and integrating by parts, we get:
	\begin{equation}\label{Bresse equation 2.15}
-\int_0^L \eta \left|\lambda v^1\right|^2dx-i \int_0^L\lambda \eta v^2 \overline{v^1}dx= i \int_0^L \frac{f^1}{\lambda^2}\eta \lambda\overline{v^1}dx.
	\end{equation}
	As $\lambda v^1$ is uniformly bounded in $L^2(0,L)$ and $f^1$ converges to zero in $H_0^1(0,L)$, we get
	%\begin{equation}
	%i \int_0^L\lambda f^1\eta \overline{v}^1dx= o\left(1\right).
	%\end{equation}
	that the term on the right hand side of \eqref{Bresse equation 2.15} converges to zero and consequently
	\begin{equation}\label{Bresse equation 2.16}
-\int_0^L \eta \left|\lambda v^1\right|^2dx-i \int_0^L\lambda \eta v^2 \overline{v^1}dx = \frac{o(1)}{\lambda^2}.
	\end{equation}
	%%%%%%%%%%%%%%%%%%%%%%%%%%%%%%%%%%%%%%%%%%%%%%%%%%%
	Moreover, multiplying \eqref{Bresse equation 2.4} by $\rho_1^{-1} \eta \overline {v^1}$ in $L^2(0,L)$, then integrating by parts we obtain:
	\begin{eqnarray}\label{Bresse equation 2.17}
i \int_0^L\lambda \eta v^2 \overline{v^1}dx+  \rho_1^{-1}\int_0^L \left(k_1\left(v^1_x+v^3+\ell  v^5\right) +D_1 \left( v^2_x+v^4+\ell v^6\right)\right)\left(\eta \overline {v^1}\right)_xdx\nonumber \\
- \ell  k_3 \rho_1^{-1}\int_0^L\left(v^5_x-\ell  v^1\right)\eta \overline {v^1}dx - \ell \rho_1^{-1}\int_0^L D_3 \left(v^6_x - \ell  v^2\right)\eta \overline {v^1}dx=\int_0^L\frac{f^2}{\lambda^2}\eta \overline {v^1}dx.
	\end{eqnarray} 
	Using \eqref{Bresse equation 2.10}, \eqref{Bresse equation 2.13}, the fact that $f^2$ converges to zero in $L^2(0,L)$ and $\lambda v^1$, $v^1_x$
are uniformly bounded in $L^2(0,L)$ in \eqref{Bresse equation 2.17}, we get:
	\begin{equation}\label{Bresse equation 2.18}
i \int_0^L\lambda \eta v^2 \overline{v}^1dx=\frac{o\left(1\right)}{\lambda}.
	\end{equation}
	Finally, using \eqref{Bresse equation 2.18} in \eqref{Bresse equation 2.16} and the definiton of $\eta$, we get: 
	\begin{equation*}
\int_{0}^{L}\eta\left|\lambda v^1\right|^2dx=\frac{o\left(1\right)}{\lambda}, \quad
\int_{\alpha+\epsilon}^{\beta-\epsilon}\left|\lambda v^1\right|^2dx=\frac{o\left(1\right)}{\lambda}.
	\end{equation*}
	%%%%%%%%%%%%%%%%%%%%%%%%%%%%%%%%%%%%%%%%%%%%%%%%%%%
	In a same way, we show:
	\begin{equation*}
\int_{\alpha+\epsilon}^{\beta-\epsilon}\left|\lambda v^3\right|^2dx=\frac{o\left(1\right)}{\lambda}, \ \
\int_{\alpha+\epsilon}^{\beta-\epsilon}\left|\lambda v^5\right|^2dx=\frac{o\left(1\right)}{\lambda}.
	\end{equation*}
	The proof is thus complete.
\end{proof}
%%%%%%%%%%%%%%%%%%%%%%%%%%%%%%%%%%%%%%%%%%%%%%%%%%%
%%%%%%%%%%%%%%%%%%%%%%%%%%%%%%%%%%%%%%%%%%%%%%%%%%%
%LEMMA
%%%%%%%%%%%%%%%%%%%%%%%%%%%%%%%%%%%%%%%%%%%%%%%%%%%%%%%%%%%%%%%%%%%%%%%%%%%%%%%%%%%%%%%%%%%%%%%%%%%%%%
Now, we introduce new multiplier functions. For this purpose, let 
$\emptyset\not=\omega_\epsilon=(\alpha+\epsilon,\beta-\epsilon).$
\begin{lem}
	The solution $(u,y,z)$ of the following system
	\begin{equation}\label{Bresse Aux System}
\left\{
\begin{matrix}
\rho_1 \lambda^2 u + k_1 \left(u_x+y+\ell z\right)_x + \ell k_3\left(z_x-\ell u\right) -i \lambda  \mathds{1}_{\omega_\epsilon} u&= &v^1,\\
\rho_2 \lambda^2 y + k_2y_{xx}-k_1\left(u_x+y+\ell z\right)- i \lambda \mathds{1}_{\omega_\epsilon} y& =&v^3, \\	
\rho_1 \lambda^2 z + k_3 \left(z_x-\ell u\right)_x - \ell k_1\left(u_x+y+\ell z\right) - i \lambda \mathds{1}_{\omega_\epsilon} z&= &v^5
\end{matrix} 
\right.
	\end{equation}
	with fully Dirichlet boundary conditions: 
	\begin{equation}\label{FD}
u\left(0\right)=u\left(L\right)=y\left(0\right)=y\left(L\right)=
z\left(0\right)=z\left(L\right)=0
	\end{equation}
	or with Dirichlet-Neumann-Neumann boundary conditions:
	\begin{equation}\label{DN}
u\left(0\right)=u\left(L\right)=y_x\left(0\right)=y_x\left(L\right)=
z_x\left(0\right)=z_x\left(L\right)=0
	\end{equation}
	verifies the following inequality: 
	\begin{align}\label{Bresse Aux Inequality}
\int_0^L \bigg(& \rho_1|\lambda u|^2 + \rho_2|\lambda y|^2 + \rho_1|\lambda z|^2 +k_2 |y_x|^2\nonumber \\
&\ \ +k_1|u_x+y +\ell z|^2 +k_3|z_x-\ell u|^2\bigg)dx {\ \leq \ } C \int_0^L \left(|v^1|^2+|v^3|^2 +|v^5|^2\right)dx,
	\end{align}
	where $C$ is a constant independent of $n$.
\end{lem}
%%%%%%%%%%%%%%%%%%%%%%%%%%%%%%%%%%%%%%%%%%%%%%%%%%%
%PROOF
%%%%%%%%%%%%%%%%%%%%%%%%%%%%%%%%%%%%%%%%%%%%%%%%%%%
\begin{proof}
	We consider the following Bresse system subject to three local viscous dampings:
	\begin{equation}\label{Bresse auxsys}
\left\{
\begin{matrix}
\rho_1 u_{tt} -k_1\left(u_x+y+\ell z\right)_x -lk_3\left(z_x-\ell u\right) + \mathds{1}_{\omega_\epsilon}u_t&=&0,\\ \\
\rho_2 y_{tt} - k_2 y_{xx} + k_1\left(u_x +y +\ell z\right)+\mathds{1}_{\omega_\epsilon}y_t&=&0,\\ \\
\rho_1z_{tt} -k_3\left(z_x-\ell u\right)_x+ \ell k_1\left(u_x+y+\ell z\right) +\mathds{1}_{\omega_\epsilon}z_t&=&0
\end{matrix}
\right.
	\end{equation}
	with fully Dirichlet or Dirichlet-Neumann-Neumann boundary conditions. 
	Systems \eqref{Bresse auxsys}-\eqref{FD} and \eqref{Bresse auxsys}-\eqref{DN} are well posed in the space $H_1= \left(H_0^1(0,L)\times L^2(0,L)\right)^3$
and in the space $H_2= \left(H_0^1(0,L)\times L^2(0,L)\right)\times  \left(H_*^1(0,L)\times L_*^2(0,L)\right)^2$ respectively.
In addition, both are exponentially stable (see \cite{wehbey}). Therefore, following Huang \cite{Huang-85} and Pruss \cite{Pruss-84}, we deduce that the resolvent of the associated operator:
	$$\AA_{aux_j} : D(\AA_{aux_j})\subset H_j\rightarrow H_j$$
	defined by 
	$$D\left(\mathcal{A}_{aux_1}\right)=\big(H_0^1(\Omega)\cap H^2(\Omega)\big)^3\times\big(H_0^1(\Omega)\big)^3,$$
	$$D\left(\mathcal{A}_{aux_2}\right)=\big\{U \in H_2: u\in H_0^1\cap H^2, y,z\in H^1_\star\cap H^2, \tilde u,y_x, z_x\in
	H^1_0,\tilde y,\tilde z\in H^1_\star\big\}$$
	and
	\begin{equation*}
\mathcal{A}_{aux_j}\left(\begin{array}{l}
u\\  \noalign{\medskip} \tilde u\\  \noalign{\medskip}y\\ \tilde y\\  \noalign{\medskip} z\\  \noalign{\medskip} \tilde z
\end{array}\right)=\left(\begin{array}{c}
\tilde u\\  \noalign{\medskip} \rho_1^{-1}\left[k_1(u_x+y+\ell z)_x    +\ell k_3(z_x-\ell u)-\mathds{1}_{\omega_\epsilon}\tilde u\right]\\  \noalign{\medskip} \tilde y\\  \noalign{\medskip} 
\rho_2^{-1} \left[k_2y_{xx}-k_1\left(u_x+y+\ell  z\right)-\mathds{1}_{\omega_\epsilon}\tilde y\right]\\  \noalign{\medskip}
\tilde z\\  \noalign{\medskip}
\rho_1^{-1}\left[k_3(z_x-\ell  u)_x-\ell  k_1\left(u_x+z+\ell  z\right)-\mathds{1}_{\omega_\epsilon}\tilde z\right] 
\end{array}\right)
	\end{equation*}
	is uniformly bounded on the imaginary axis. So, by setting $ \tilde u=i \lambda u $, $ \tilde y=i \lambda y $ and $ \tilde z=i \lambda z $, we deduce that: 
	\begin{equation*}
\left(\begin{array}{l}
u\\  \noalign{\medskip} \tilde u\\  \noalign{\medskip}y\\ \tilde y\\  \noalign{\medskip} z\\  \noalign{\medskip} \tilde z
\end{array}\right)=	\left(i\lambda-\mathcal{A}_{aux_j}\right)^{-1}\left(\begin{array}{c}
0\\  \noalign{\medskip} \frac{-1}{\rho_1}v^1\\  \noalign{\medskip}0\\ \frac{-1}{\rho_2}v^3\\  \noalign{\medskip} 0\\  \noalign{\medskip} \frac{-1}{\rho_1}v^1
\end{array}\right).
	\end{equation*}
	%%%%%%%%%%%%%%%%%%%%%%%%%%%%%%%%%%%%%%%%%%%%%%%%%%%
	This yields:
	\begin{align}
\left\|\left(u, \tilde u, y, \tilde y, z, \tilde z\right)\right\|^2_{H_j} &\leq \norm{	\left(i\lambda-\mathcal{A}_{aux_j}\right)^{-1}}_{\mathcal{L}(H_j)} \norm{(0, \frac{-1}{\rho_1}v^1, 0, \frac{-1}{\rho_2}v^3, 0,\frac{-1}{\rho_1}v^5)}_{H_j}\nonumber\\
&\leq C \int_0^L \left(|v^1|^2+|v^3|^2 +|v^5|^2\right)dx,
	\end{align}
	where $C$ is a constant independent of $n$. 
	Consequently, \eqref{Bresse Aux Inequality} holds. The proof is thus complete. 
\end{proof}
%%%%%%%%%%%%%%%%%%%%%%%%%%%%%%%%%%%%%%%%%%%%%%%%%%%%%%%%%%%%%%%%%%%%%%%%%%%%%%%%%%%%%%%%%%%%%%%%%%%%%
%LEMMA
%%%%%%%%%%%%%%%%%%%%%%%%%%%%%%%%%%%%%%%%%%%%%%%%%%%%%%%%%%%%%%%%%%%%%%%%%%%%%%%%%%%%%%%%%%%%%%%%%%%%%%
\begin{lem}\label{lemma6.6}
	Under all the above assumptions, we have:
	\begin{equation}\label{Bresse equation 2.32}
\int_0^L|\lambda v^1|^2dx=o(1), \ \ 
\int_0^L|\lambda v^3|^2dx=o(1), \ \
\int_0^L|\lambda v^5|^2dx=o(1).
	\end{equation}
\end{lem}
%%%%%%%%%%%%%%%%%%%%%%%%%%%%%%%%%%%%%%%%%%%%%%%%%%
%PROOF
%%%%%%%%%%%%%%%%%%%%%%%%%%%%%%%%%%%%%%%%%%%%%%%%%%%
\begin{proof}
	For clarity of the proof, we divide the proof into several steps.\\
	\textbf{Step 1.} First, multiplying \eqref{Bresse equation 2.3} by $i \rho_1 \lambda \overline u$, where $u$ is a solution of system \eqref{Bresse Aux System}, we get: 
	\begin{equation}\label{Bresse equation 2.19}
-\int_0^L \rho_1\lambda^2 \overline u v^1dx - i \int_0^L\rho_1 \lambda \overline u v^2 dx=\rho_1 \int_0^L\frac{ if^1}{\lambda}\overline udx.
	\end{equation}
	Moreover, multiplying \eqref{Bresse equation 2.4}  by $\overline u$ and integrating by parts, we obtain:
	\begin{align}\label{Bresse equation 2.20}
& i \int_0^L \rho_1 \lambda \overline u v^2 dx - \int_0^Lk_1 \overline u_{xx}v^1dx- \int_0^L \ell k_3	(-\ell \overline u)v^1dx + \int_0^Lk_1\overline u_xv^3dx+ \int_0^L \ell k_1\overline u_x v^5dx\nonumber\\
& +\int_0^L \ell  k_3 \overline u_x v^5dx + \int_0^LD_1(v^2_x+v^4 +\ell v^6)\overline u_xdx -\int_0^L\ell D_3(v^6_x-\ell v^2)\overline u dx=\rho_1 \int_0^L \frac{f^2}{\lambda^2}\overline udx.
	\end{align}	
	Now, combining \eqref{Bresse equation 2.19} and \eqref{Bresse equation 2.20}, we get:
	\begin{align}\label{Bresse equation 2.21}
&\int_0^L\left[\rho_1 \lambda^2 \overline u +k_1 \overline u_{xx}+\ell k_3(-\ell \overline u)\right]v^1dx - \int_0^Lk_1\overline u_xv^3dx- \int_0^L \ell k_1\overline u_x v^5dx -\int_0^L \ell  k_3 \overline u_x v^5dx \nonumber\\
& - \int_0^LD_1(v^2_x+v^4 +\ell v^6)\overline u_xdx +\int_0^L\ell D_3(v^6_x-\ell v^2)\overline u dx=-\rho_1 \int_0^L\left(\frac{i f^1}{\lambda} +\frac{f^2}{\lambda^2}\right)\overline udx.
	\end{align}	
	%%%%%%%%%%%%%%%%%%%%%%%%%%%%%%%%%%%%%%%%%%%%%%%%%%%
	\textbf{Step 2.} Similarly to Step 1, multiplying \eqref{Bresse equation 2.5} by $i\rho_2 \lambda \overline y$ and \eqref{Bresse equation 2.6}
by $\overline y$, where $y$ is a solution of system \eqref{Bresse Aux System}, we get: 	
	\begin{align}\label{Bresse equation 2.22}
&\int_0^L\left[\rho_2 \lambda^2 \overline y+k_2 \overline y_{xx} -k_1\overline y\right]v^3dx+ \int_0^Lk_1\overline y_x v^1dx -\int_0^L\ell k_1\overline y v^5dx\nonumber\\ &-\int_0^LD_2v^4_x \overline y_xdx-\int_0^LD_1(v^2_x+v^4+\ell v^6)\overline ydx=-\rho_2\int_0^L\left(\frac{i f^3}{\lambda}+\frac{f^4}{\lambda^2}\right)\overline ydx.	
	\end{align}
	%%%%%%%%%%%%%%%%%%%%%%%%%%%%%%%%%%%%%%%%%%%%%%%%%%%
	\textbf{Step 3.} As in Step 1 and Step 2, by multiplying  \eqref{Bresse equation 2.7} by $i\rho_1 \lambda \overline z$ and \eqref{Bresse equation 2.8}
by $\overline z$, where $z$ is a solution of system \eqref{Bresse Aux System}, we get:
	\begin{align}\label{Bresse equation 2.23}
&\int_0^L\left[\rho_1 \lambda^2 \overline z +k_3 \overline z_{xx} - \ell  k_1\left(\ell \overline z\right)\right]v^5dx +\int_0^L \ell k_3\overline z_x v^1dx +\int_0^L\ell k_1 \overline z_x v^1dx \nonumber\\
&-\int_0^L \ell k_1\overline z v^3dx -\int_0^LD_3\left(v^6_x-\ell v^2\right) \overline z_xdx
-\ell \int_0^LD_1(v^2_x+v^4+\ell v^6)\overline z dx = - \rho_1\int_0^L\left( \frac{if^5}{\lambda} +\frac{f^6}{\lambda^2}\right)\overline zdx.
	\end{align}
	%%%%%%%%%%%%%%%%%%%%%%%%%%%%%%%%%%%%%%%%%%%%%%%%%%
	\textbf{Step 4.} First, combining \eqref{Bresse equation 2.21}, \eqref{Bresse equation 2.22} and \eqref{Bresse equation 2.23}, we obtain:
	\begin{align}\label{Bresse equation 2.24}
&\int_0^L\left[\rho_1 \lambda^2 \overline u +k_1 \left(\overline u_{x}+\overline y + \ell \overline z\right)_x+\ell k_3(\overline z_x-\ell \overline u)\right]v^1dx
+\int_0^L\left[\rho_2 \lambda^2 \overline y+k_2 \overline y_{xx} -k_1\left(\overline u_x+\overline y + \ell \overline z \right)\right]v^3dx	\nonumber\\
&+\int_0^L\left[\rho_1 \lambda^2 \overline z +k_3\left( \overline z_{x} -\ell \overline u\right)_x - \ell  k_1\left(\overline u_x +\overline y+ \ell \overline z\right)\right]v^5dx - \int_0^LD_1(v^2_x+v^4 +\ell v^6)\overline u_xdx\nonumber \\ &+\int_0^L\ell D_3(v^6_x-\ell v^2)\overline u dx	-\int_0^LD_2v^4_x \overline y_xdx-\int_0^LD_1(v^2_x+v^4+\ell v^6)\overline ydx -\int_0^LD_3\left(v^6_x-\ell v^2\right) \overline z_xdx\\
&-\ell \int_0^LD_1(v^2_x+v^4+\ell v^6)\overline z dx = -\rho_1 \int_0^L\left(\frac{i f^1}{\lambda} +\frac{f^2}{\lambda^2}\right)\overline udx -\rho_2\int_0^L\left(\frac{i f^3}{\lambda}+\frac{f^4}{\lambda^2}\right)\overline ydx \nonumber\\
&- \rho_1\int_0^L\left( \frac{if^5}{\lambda} +\frac{f^6}{\lambda^2}\right)\overline zdx.\nonumber
	\end{align}
	Combining equation \eqref{Bresse Aux System} and \eqref{Bresse equation 2.24}, multiplying by $\lambda^2$, we get:   

	\begin{align}\label{Bresse equation 2.26}
&\int_0^L|\lambda v^1|^2dx + \int_0^L|\lambda v^3|^2dx + \int_0^L|\lambda v^5|^2dx = i\int_{\alpha +\epsilon}^{\beta-\epsilon}(\lambda^2 \overline u\lambda v^1 dx  + \lambda^2 \overline y\lambda v^3+ \lambda^2 \overline z \lambda v^5) dx  \nonumber \\
& +\int_0^L \lambda D_1(v^2_x+v^4 +\ell v^6)\lambda \overline u_xdx -\int_0^L\ell D_3(v^6_x-\ell v^2)\lambda^2\overline u dx	+\int_0^L\lambda D_2v^4_x \lambda \overline y_xdx \nonumber\\
&+\int_0^L D_1(v^2_x+v^4+\ell v^6) \lambda^2\overline ydx +\int_0^L\lambda D_3\left(v^6_x-\ell v^2\right) \lambda \overline z_xdx
+\ell \int_0^LD_1(v^2_x+v^4+\ell v^6)\lambda^2\overline z dx \\
&-\rho_1 \int_0^L\left(i f^1 \lambda + f^2 \right)\overline u dx -\rho_2 \int_0^L\left(i f^3 \lambda + f^4 \right)\overline y dx -\rho_1 \int_0^L\left(i f^5 \lambda + f^6 \right)\overline z dx. \nonumber 
	\end{align}
	Using  {estimates} \eqref{Bresse equation 2.14} 
	and the fact that $\lambda^2 u$, $\lambda^2 y$ and $\lambda^2 z$ are uniformly bounded in $L^2(0,L)$ due to \eqref{Bresse Aux Inequality}, we get:
	\begin{equation}\label{Bresse equation 2.28}
i\int_{\alpha +\epsilon}^{\beta-\epsilon}(\lambda^2 \overline u\lambda v^1 dx  + \lambda^2 \overline y\lambda v^3+ \lambda^2 \overline z \lambda v^5) dx = \frac{o(1)}{\lambda^{1/2}}.  
	\end{equation}
	In addition, using %Lemma \eqref{Bresse polydiss.1}
	\eqref{Bresse equation 2.10} and the fact that $\lambda u_x$, $\lambda y_x$ and $\lambda z_x$ are uniformly bounded in $L^2(0,L)$ due to \eqref{Bresse Aux Inequality}. we get: 
	\begin{equation}\label{Bresse equation 2.29}
\int_0^L \lambda D_1(v^2_x+v^4 +\ell v^6)\lambda \overline u_xdx+\int_0^L\lambda D_2v^4_x \lambda \overline y_xdx \int_0^L\lambda D_3\left(v^6_x-\ell v^2\right) \lambda \overline z_xdx = o(1).
	\end{equation}
	Also, by using %Lemma \eqref{Bresse polydiss.1}
	\eqref{Bresse equation 2.10} and the fact that $\lambda^2 u$, $\lambda^2 y$ and $\lambda^2 z$ are uniformly bounded in $L^2(0,L)$ due to \eqref{Bresse Aux Inequality}, we obtain: 
	\begin{equation}\label{Bresse equation 2.30}
\int_0^L\ell D_3(v^6_x-\ell v^2)\lambda^2\overline u dx+\int_0^L D_1(v^2_x+v^4+\ell v^6) \lambda^2\overline ydx 
+\ell \int_0^LD_1(v^2_x+v^4+\ell v^6)\lambda^2\overline z dx =\frac{o(1)}{\lambda}.
	\end{equation}
	Moreover, we have:
	\begin{equation}\label{Bresse equation 2.31}
-\rho_1 \int_0^L\left(i f^1 \lambda + f^2 \right)\overline u dx -\rho_2 \int_0^L\left(i f^3 \lambda + f^4 \right)\overline y dx -\rho_1 \int_0^L\left(i f^5 \lambda + f^6 \right)\overline z dx=o(1),
	\end{equation}
	since $f^1$, $f^3$, $f^5$ converge to zero in $H_0^1(0,L)$ (or in $H_\star^1(0,L)$), $f^2$, $f^4$, $f^6$ converge to zero in $L^2(0,L)$, and  $\lambda^2 u$, $\lambda^2 y$, $\lambda^2 z$ are uniformly bounded in $L^2(0,L)$.\\
	%%%%%%%%%%%%%%%%%%%%%%%%%%%%%%%%%%%%%%%%%%%%%%%%%%
	Finally, inserting \eqref{Bresse equation 2.28} -
	\eqref{Bresse equation 2.31} into \eqref{Bresse equation 2.26}, we get the desired  {estimates}  in \eqref{Bresse equation 2.32}. Thus the proof is complete.	
\end{proof}	
%%%%%%%%%%%%%%%%%%%%%%%%%%%%%%%%%%%%%%%%%%%%%%%%%%%%%%%%%%%%%%%%%%%%%%%%%%%%%%%%%%%%%%%%%%%%%%%%%%%%
%LEMMA
%%%%%%%%%%%%%%%%%%%%%%%%%%%%%%%%%%%%%%%%%%%%%%%%%%%%%%%%%%%%%%%%%%%%%%%%%%%%%%%%%%%%%%%%%%%%%%%%%%%%
\begin{lem}\label{lemma6.7}
	Under all the above assumptions, we have:
	\begin{equation}\label{Bresse equation 2.33}
\int_0^L| v^1_x|^2dx=o(1), \ \ 
\int_0^L| v^3_x|^2dx=o(1), \ \
\int_0^L| v^5_x|^2dx=o(1).
	\end{equation}
\end{lem}
%%%%%%%%%%%%%%%%%%%%%%%%%%%%%%%%%%%%%%%%%%%%%%%%%&
%PROOF
%%%%%%%%%%%%%%%%%%%%%%%%%%%%%%%%%%%%%%%%%%%%%%%%%%
\begin{proof}
	First, multiplying \eqref{Bresse equation 2.4} by $\overline {v^1}$ and then integrating by parts, we get:
	\begin{align}\label{Bresse equation 2.34}
&i\int_0^L\rho_1\lambda v^2 \overline{ v^1}dx +k_1\int_0^L| v^1_x|^2dx+k_1\int_0^L\left(v^3+\ell  v^5\right)\overline {v^1_x}dx +\int_0^LD_1  \left(v^2_x+v^4+\ell v^6\right)\overline{v_x^1}dx\nonumber \\
&-  \ell  k_3\int_0^L\left(v^5_x-\ell  v^1\right)\overline{v^1}dx - \ell \int_0^L D_3 \left(v^6_x - \ell  v^2\right)\overline{v^1}dx= \rho_1\int_0^L\frac{f^2}{\lambda^2}\overline{v^1}dx.
	\end{align}
	Then, using \eqref{Bresse equation 2.9}, \eqref{Bresse equation 2.10} and the fact that $v_x^1$, $\left(v^5_x-\ell  v^1\right)$ are uniformly bounded in $L^2(0,L)$ due to \eqref{Bresse equation 2.1}, we obtain: 
	\begin{align}\label{Bresse equation 2.35}
&k_1\int_0^L\left(v^3+\ell  v^5\right)\overline {v^1_x}dx +\int_0^LD_1  \left(v^2_x+v^4+\ell v^6\right)\overline{v_x^1}dx
\nonumber\\
&-  \ell  k_3\int_0^L\left(v^5_x-\ell  v^1\right)\overline{v^1}dx 
- \ell \int_0^L D_3 \left(v^6_x - \ell  v^2\right)\overline{v^1}dx=o(1).
	\end{align}
	As $f^2$ converges to zero in $L^2(0,L)$ and $\lambda v^1$ is uniformly bounded in $L^2(0,L)$, we have:
	\begin{equation}\label{Bresse equation 2.36}
\rho_1\int_0^L\frac{f^2}{\lambda^2}\overline{v^1}dx=o(1).
	\end{equation}
	Next, inserting \eqref{Bresse equation 2.35} and 
	\eqref{Bresse equation 2.36} into \eqref{Bresse equation 2.34}, we get: 
	\begin{align}\label{Bresse equation 2.37}
i\int_0^L\rho_1\lambda v^2 \overline{ v^1}dx +k_1\int_0^L| v^1_x|^2dx=o(1).
	\end{align}
	Using Lemma \ref{lemma6.6} and the fact that $ v^2$ is uniformly bounded in $L^2(0,L)$ due to \eqref{Bresse equation 2.37}, we deduce:
	\begin{equation*}
\int_0^L| v^1_x|^2dx=o(1).
	\end{equation*}
	Similarly, one can prove that:
	\begin{equation*}
\int_0^L| v^3_x|^2dx=o(1), \ \ \int_0^L| v^5_x|^2dx=o(1).
	\end{equation*}
	Thus, the proof is complete.
\end{proof}
%%%%%%%%%%%%%%%%%%%%%%%%%%%%%%%%%%%%%%%%%%%%%%%%%%
% PROOF OF THE THEOREM 
%%%%%%%%%%%%%%%%%%%%%%%%%%%%%%%%%%%%%%%%%%%%%%%%%
%\textbf{ Proof of Theorem \ref{PolyThe.2}} 
\begin{proof}[ Proof of Theorem \ref{PolyThe.2}]Using Lemma \ref{lemma6.6} and  Lemma \ref{lemma6.7}, we get  that $\norm{U}_{\HH_j}=o(1)$.
Therefore, we get a contradiction with \eqref{Bresse equation 2.1} and consequently (H3) holds. Thus the proof is complete
\end{proof}
\begin{rk}
It is known that for a single one-dimensional wave equation with damping coefficient $D_{1} = d_{0} > 0 $ on $\omega$, the optimal solution decay rate is $1/t^2$.
The new multipliers (one for each equation) we have used here, defined by system \eqref{Bresse Aux System}, do not permit to obtain a decay rate of  $1/t^2$ but only $1/t$.
This may be due to the coupling effects and we do not know if this decay rate of $1/t$ is optimal.
\end{rk}
%%%%%%%%%%%%%%%%%%%%%%%%%%%%%%%%%%%%%%%%%%%%%%%%%%%%%%%%%%%%%%%%%%%%%%%%%%%%%%%%%%%%%%%%%%%%%%%%%%%%%%%%%%%%%%
% POLYNOMIAL STABILITY
%%%%%%%%%%%%%%%%%%%%%%%%%%%%%%%%%%%%%%%%%%%%%%%%%%%%%%%%%%%%%%%%%%%%%%%%%%%%%%%%%%%%%%%%%%%%%%%%%%%%%%%%%%%%%%
\section{The case of only one local viscoelastic damping with non smooth coefficient at the interface} \label{Polynomial2}
In control theory, it is important to reduce the number of {control such as damping terms}.
So, this section is devoted to show the polynomial stability of systems \eqref{eqq1.1'}-\eqref{DDDD} and \eqref{eqq1.1'}-\eqref{DNND}
subject to only one viscoelastic Kelvin-Voigt damping with non smooth coefficient at the interface. For this purpose, we consider the following condition: 
\begin{equation}\label{onelocal}
	D_1=D_3=0 \, \, \mathrm{in}\, \,(0,L) \quad \mathrm{and} \quad \quad \exists \, \, d_0>0\,\, \mbox{such that}\,\, D_2 \geq d_0>0 \, \, \mathrm{in} \, \, \emptyset \not=(\alpha,\beta) \subset (0,L).
\end{equation}
The main result of this section is given by the following theorem:
%%%%%%%%%%%%%%%%%%%%%%%%%%%%%%%%%%%%%%%%%%%%%%%%%%%%%%%
% THEOREM
%%%%%%%%%%%%%%%%%%%%%%%%%%%%%%%%%%%%%%%%%%%%%%%%%%%%%%%
\begin{thm}\label{PolyThe1.2}
	Assume that condition \eqref{onelocal} is satisfied. Then, there exists a positive constant $c>0$ such that for all $U_0 \in D{(\AA_j)}$, $j=1,2,$
the energy of system \eqref{eqq1.1'} satisfies the following decay rate:
	\begin{equation}\label{polydecay2.2}
E(t)\leq  \frac{c}{\sqrt{t}}\norm{U_0}^2_{D{(\AA_j)}}.
	\end{equation}
\end{thm}
%%%%%%%%%%%%%%%%%%%%%%%%%%%%%%%%%%%%%%%%%%%%%%%%%%%%%%%
%PROOF
%%%%%%%%%%%%%%%%%%%%%%%%%%%%%%%%%%%%%%%%%%%%%%%%%%%%%%%
Referring to \cite{Borichev-Tomilov10}, \eqref{polydecay2.2} is verified if the following conditions
\begin{equation*}\tag {\rm H1}
	\ i\mathbb{R}\subseteq\rho\left(\mathcal{A}_j\right) 
\end{equation*}
and
\begin{equation*}\tag {\rm H4}
	\lim_{|\lambda|\rightarrow +\infty}\sup_{\lambda\in\mathbb{R}}\bigg\{\frac{1}{\lambda^4}\left\|\left(i\lambda I-\mathcal{A}_j\right)^{-1}\right\|_{\mathcal{L}\left(\mathcal{H}_j\right)}\bigg\}=O\left(1\right)
\end{equation*}
hold. \\
%%%%%%%%%%%%%%%%%%%%%%%%%%%%%%%%%%%%%%%%%%%%%%%%%%
Condition  $\ i\mathbb{R}\subseteq\rho\left(\mathcal{A}_j\right)$ is  already proved in Lemma \ref{Bresse Lemma 2.6} and Lemma \ref{Bresse Lemma 2.7}. 

We will establish (H4) by  contradiction. Suppose that there exist a sequence of real numbers $\left(\lambda_n\right)_n$, with  $|\lambda_n|\to+\infty$  and a sequence of vectors
\begin{equation}\label{Bresse Equation 1.8}
	U_n=\left(v^1_n,v^2_n,v_n^3,v^4_n,v^5_n,v^6_n\right)^{\mathsf{T}}\in D\left(\mathcal{A}_j\right) \ \text{ with }\ \|U_n\|_{\mathcal{H}_j}=1
\end{equation}
%%%%%%%%%%%%%%%%%%%%%%%%%%%%%%%%%%%%%%%%%%%%%%%%%%
such that
\begin{equation}\label{Bresse Equation 2.8}
	\lambda_n^4\left(\ i\lambda_nU_n-\mathcal{A}_jU_n\right)=\left(f^1_n,f^2_n,f_n^3,f^4_n,f^5_n,f^6_n\right)^{\mathsf{T}}\to 0\ \text{ in } \mathcal{H}_j,\quad j=1,2.
\end{equation}
%%%%%%%%%%%%%%%%%%%%%%%%%%%%%%%%%%%%%%%%%%%%%%%%%%
We will check the condition (H4) by finding a contradiction with \eqref{Bresse Equation 1.8}-\eqref{Bresse Equation 2.8} such as $\left\| U_n\right\|_{\mathcal{H}_j} =o(1)$. 
%%%%%%%%%%%%%%%%%%%%%%%%%%%%%%%%%%%%%%%%%%%%%%%%%%

Equation \eqref{Bresse Equation 2.8} is detailed as:
\begin{eqnarray}
	\ i\lambda_n v^1_n-v^2_n&=&\frac{f^1_n}{\lambda_n^4},\label{Bresse Equation 3.8}
	\\
	i\rho_1\lambda_n v^2_n-k_1 \left[\left(v^1_n\right)_x+v^3_n+\ell  v^5_n\right]_{x}-  \ell  k_3\left[\left(v^5_n\right)_x-\ell  v^1_n\right]& =& \rho_1\frac{f^2_n}{\lambda_n^4},\label{Bresse Equation 4.8}
	\\
	\ i\lambda_n v^3_n-v^4_n&=&\frac{f^3_n}{\lambda_n^4},\label{Bresse Equation 5.8}
	\\
	\ i \rho_2 \lambda_n v^4_n -\left[k_2 \left(v^3_n\right)_x + D_2 \left(v^4_n\right)_x\right]_{x}+ k_1\left[\left(v^1_n\right)_x+v^3_n+\ell  v^5_n\right]
	&=& \rho_2\frac{f^4_n}{\lambda_n^4},\label{Bresse Equation 6.8}
	\\
	i\lambda_n v^5_n-v^6_n&=& \frac{f^5_n}{\lambda_n^4},\label{Bresse equation 7.8}\\
	\ i \rho_1 \lambda_n v^6_n-\left[k_3\left(\left(v^5_n\right)_x-\ell  v^1_n\right) \right]_x + \ell  k_1\left[\left(v^1_n\right)_x+v^3_n+\ell  v^5_n\right]
	&=&\rho_1 \frac{f^6_n}{\lambda_n^4}.\label{Bresse Equation 8.8}
\end{eqnarray}
%%%%%%%%%%%%%%%%%%%%%%%%%%%%%%%%%%%%%%%%
%%%%%%%%%%%%%%%%%%%%%%%%%%%%%%%%%%%
Inserting \eqref{Bresse Equation 3.8}, \eqref{Bresse Equation 5.8},
and \eqref{Bresse equation 7.8} into \eqref{Bresse Equation 4.8},\eqref{Bresse Equation 6.8} and \eqref{Bresse Equation 8.8} respectively, we get
\begin{eqnarray}
	\rho_1\lambda^2_n v^1_n+k_1 \left[\left(v^1_n\right)_x+v^3_n+\ell  v^5_n\right]_{x}+ \ell  k_3\left[\left(v^5_n\right)_x-\ell  v^1_n\right]& =& -i \rho_1 \frac{f_n^1}{\lambda_n^3}-\rho_1\frac{f^2_n}{\lambda_n^4},\label{Bresse Equation a.8}
	\\
	\rho_2 \lambda_n^2 v^3_n +\left[k_2 \left(v^3_n\right)_x + D_2 \left(v^4_n\right)_x\right]_{x}- k_1\left[\left(v^1_n\right)_x+v^3_n+\ell  v^5_n\right]
	&=&-i \rho_2 \frac{f_n^3}{\lambda_n^3} -\rho_2\frac{f^4_n}{\lambda_n^4},\label{Bresse Equation b.8}
	\\
	\rho_1 \lambda^2_n v^5_n+\left[k_3\left(\left(v^5_n\right)_x-\ell  v^1_n\right) \right]_x - \ell  k_1\left[\left(v^1_n\right)_x+v^3_n+\ell  v^5_n\right]
	&=&-i \rho_1\frac{f_n^5}{\lambda_n^3}-\rho_1 \frac{f^6_n}{\lambda_n^4}.\label{Bresse Equation c.8}
\end{eqnarray}
%%%%%%%%%%%%%%%%%%%%%%%%%%%%%%%%%%%%%%%%%%%%%%%%%%%
%%%%%%%%%%%%%%%%%%%%%%%%%%%%%%%%%%%%%%%%%%%%%%%%%%%
From \eqref{Bresse Equation 3.8}, \eqref{Bresse Equation 5.8}, \eqref{Bresse equation 7.8} and \eqref{Bresse Equation 1.8}, we deduce that:
\begin{equation}\label{Bresse Equation 9.8}
	\norm{v^1_n}=O(\frac{1}{\lambda_n}), \ \  \norm{v^3_n}=O(\frac{1}{\lambda_n}), \ \
	\norm{v^5_n}=O(\frac{1}{\lambda_n}).
\end{equation}
%%%%%%%%%%%%%%%%%%%%%%%%%%%%%%%%%%%%%%%%%%%%%%%%%%%
For clarity, we divide the proof into several lemmas. From  now on, for simplicity, we drop the index $n$.
%%%%%%%%%%%%%%%%%%%%%%%%%%%%%%%%%%%%%%%%%%%%%%%%%%%
%LEMMA
%%%%%%%%%%%%%%%%%%%%%%%%%%%%%%%%%%%%%%%%%%%%%%%%%%%
%%%%%%%%%%%%%%%%%%%%%%%%%%%%%%%%%%%%%%%%%%%%%%%%%%	
\begin{lem}\label{Bresse polydiss2.1}
	Under all the above assumptions, we have:
	\begin{equation}\label{Bresse Equation 10.8} 
\int_{0}^{L}D_2| v^4_{x}|^2dx=\frac{o\left(1\right)}{\lambda^4}, \quad
\int_{\alpha}^{\beta}{|v^4_{x}|^2}dx=\frac{o\left(1\right)}{\lambda^4}
	\end{equation}
	and
	\begin{equation}\label{Bresse Equation 12.8} 
\int_{0}^{L}\eta| v^3_{x}|^2dx=\frac{o\left(1\right)}{\lambda^6}, \quad
%\mathrm{and} \quad 
\int_{\alpha}^{\beta}|v^3_{x}|^2dx=\frac{o\left(1\right)}{\lambda^6}.
	\end{equation}
\end{lem}
%%%%%%%%%%%%%%%%%%%%%%%%%%%%%%%%%%%%%%%%%%%%%%%%%%%
%PROOF
%%%%%%%%%%%%%%%%%%%%%%%%%%%%%%%%%%%%%%%%%%%%%%%%%%%
\begin{proof}Taking the inner product of \eqref{Bresse Equation 2.8} with $U$ in $\mathcal{H}_j$, we get:
	\begin{align}\label{Bresse Equation 11.8}
\mathrm{Re}\left(\ i\lambda^5 \norm{U}^2-\lambda^4\left(\mathcal{A}_jU,U\right)\right)_{\mathcal{H}_j}&=-\lambda^4\mathrm{Re}\left(\mathcal{A}_jU,U\right)_{\mathcal{H}_j} 
=\lambda^4\int_0^L D_2|v^4_{x}|^2 dx= o\left(1\right).
	\end{align}
	%%%%%%%%%%%%%%%%%%%%%%%%%%%%%%%%%%%%%%%%%%%%%%%%%%%
	Thanks to \eqref{onelocal}, we obtain the desired asymptotic equation \eqref{Bresse Equation 10.8}. \\
	Next, differentiating equation \eqref{Bresse Equation 5.8}, we get: 
	$$i \lambda v^3_x= v^4_x+\frac{f_x^3}{\lambda^4},$$
	and consequently
	$$\int_{\alpha}^{\beta}|\lambda v^3_x|^2dx \leq 2\int_{\alpha}^{\beta}|v^4_x|^2dx + 2\int_{\alpha}^{\beta}\frac{|f^3_x|^2}{\lambda^8}dx.$$
	Using \eqref{Bresse Equation 10.8} and the fact that $f^3$ converges to zero in $H_0^1(0,L)$ (or in $H_*^1(0,L)$ ) in the above equation,
we get the desired  {estimate} \eqref{Bresse Equation 12.8}. Thus the proof is complete.
\end{proof}
%%%%%%%%%%%%%%%%%%%%%%%%%%%%%%%%%%%
{
	\begin{rk}
Again, these estimates are crucial for the rest of the proof and they will be used to prove each point of the global proof  divided in several lemmas.
	\end{rk}
}

%%%%%%%%%%%%%%%%%%%%%%%%%%%%%%%%%%%%%%%%%%%%%%%%%%%%%%
%Let $\epsilon$ be a positive constant such that $0 <\alpha+ \epsilon <\beta-\epsilon$. We define the cut-off function $\eta$ by
%$$\eta(x)=1 \, \, \mathrm{in} \, \, (\alpha+ \epsilon, \beta-\epsilon), \quad 0 \leq \eta(x) \leq 1, \quad \eta(x)=0 \, \, \mathrm{in} \, \, (0,1)\setminus (\alpha, \beta).$$

%%%%%%%%%%%%%%%%%%%%%%%%%%%%%%%%%%%%%%%%%%%%%%%%%%%%%%
%LEMMA
%%%%%%%%%%%%%%%%%%%%%%%%%%%%%%%%%%%%%%%%%%%%%%%%%%%%%%
\begin{lem}\label{Lemma  Bresse Poly.8}
Under all the above assumptions, we have:
\begin{equation}\label{Bresse Equation 19.8} 
\int_0^L\eta|\lambda v^3|^2dx=\frac{O\left(1\right)}{\lambda^2}
\quad
\mathrm{and}
\quad
\int_{\alpha+\epsilon}^{\beta-\epsilon}|\lambda v^3|^2dx=\frac{O\left(1\right)}{\lambda^2}.
	\end{equation}
\end{lem}
%%%%%%%%%%%%%%%%%%%%%%%%%%%%%%%%%%%%%%%%%%%%%%%%%%%%%%
%PROOF
%%%%%%%%%%%%%%%%%%%%%%%%%%%%%%%%%%%%%%%%%%%%%%%%%%%%%%
\begin{proof}
	First, multiplying \eqref{Bresse Equation b.8} by $\rho_2^{-1}\eta\overline{v^3}$ and integrating by parts, we get:
	\begin{align} \label{Bresse Equation 22.8} 
\int_0^L \eta |\lambda v^3|^2dx=
&\, \rho_2^{-1}\int_0^L\left(k_2 v_x^3+D_2 v_x^4\right)\left(\eta^{\prime}\overline{v^3}+\eta \overline{v^3_x}\right)dx \nonumber 
+\rho_2^{-1}\int_0^L k_1\left(v_x^1+v^3+\ell v^5\right)\eta\overline{v^3}dx\\
&-\int_0^Li \eta \frac{f^3}{\lambda^3}\overline{v^3}dx
- \int_{0}^L \frac{f^4}{\lambda^4}\eta \overline{v^3}dx.
	\end{align}
	%%%%%%%%%%%%%
	Then, using \eqref{Bresse Equation 10.8}, \eqref{Bresse Equation 12.8}, $\norm{v^3}=O(\frac{1}{\lambda})$
and the fact that $f^3$, $f^4$ converge to zero in $H_0^1(0,L)$ (or in $H_\star^1(0,L)$), $L^2(0,L)$ respectively, we deduce that:
	\begin{align}\label{Bresse Equation 23.8} 
&\rho_2^{-1}\int_0^L\left(k_2 v_x^3+D_2 v_x^4\right)\left(\eta^{\prime}\overline{v^3}+\eta \overline{v^3_x}\right)dx 
-\int_0^Li \eta \frac{f^3}{\lambda^3}\overline{v^3}dx
- \int_{0}^L \frac{f^4}{\lambda^4}\eta\overline{v^3}dx=\dfrac{o(1)}{\lambda^3}.
	\end{align} 
	Next, inserting \eqref{Bresse Equation 23.8} into \eqref{Bresse Equation 22.8}, we obtain: 
	\begin{equation*}
\int_0^L \eta |\lambda v^3|^2dx=\rho_2^{-1}\int_0^L k_1\frac{\left(v_x^1+v^3+\ell v^5\right)}{\lambda}\eta\lambda\overline{v^3}dx+\dfrac{o(1)}{\lambda^3}. 
	\end{equation*}
	Using Cauchy-Shwartz and Young's inequalities in the above equation, we get:
	\begin{equation*}
\int_0^L \eta |\lambda v^3|^2dx \leq 2|\rho_2^{-1}|^2\int_0^Lk_1^2\eta \frac{\left|v_x^1+v^3+\ell v^5\right|^2}{\lambda^2}dx+\frac{1}{2}\int_0^L \eta |\lambda v^3|^2dx +\dfrac{o(1)}{\lambda^3},
	\end{equation*}
	Consequently, 
	\begin{equation*}
\frac{1}{2}\int_0^L \eta |\lambda v^3|^2dx \leq 2|\rho_2^{-1}|^2\int_0^Lk_1^2\eta \frac{\left|v_x^1+v^3+\ell v^5\right|^2}{\lambda^2}dx +\dfrac{o(1)}{\lambda^3}.
	\end{equation*}
	Finally, using the fact that $\left(v_x^1+v^3+\ell v^5\right)$ is uniformly bounded in $L^2(0,L)$ and the definition of $\eta$, we get the desired estimates in \eqref{Bresse Equation 19.8} and
	the proof is thus complete.
\end{proof}
%%%%%%%%%%%%%%%%%%%%%%%%%%%%%%%%%%%%%%%%%%%%%%%%%%%%%%%%%%%%%%%%%%%%%%%%%%%%%%%%%%%%%%%%%%%%%%%%%%%%%%%%%%%%%%%%%%
%LEMMA
%%%%%%%%%%%%%%%%%%%%%%%%%%%%%%%%%%%%%%%%%%%%%%%%%%%%%%%%%%%%%%%%%%%%%%%%%%%%%%%%%%%%%%%%%%%%%%%%%%%%%%%%%%%%%%%%%%
\begin{lem}\label{lemma Bresse polynomial1.8}
	Under all the above assumptions, we have:
	\begin{equation}\label{Bresse Equation 24a.8} 
\int_0^L\eta |v_x^1|^2dx=o(1), \quad 
\int_{\alpha+\epsilon}^{\beta-\epsilon} |v_x^1|^2dx=o(1)
	\end{equation}
	and 
	\begin{equation}\label{Bresse Equation 24b.8}
\int_0^L\eta |\lambda v^1|^2dx=
o(1), \quad  
\int_{\alpha+\epsilon}^{\beta-\epsilon} |\lambda v^1|^2dx=o(1).
	\end{equation}
\end{lem}
%%%%%%%%%%%%%%%%%%%%%%%%%%%%%%%%%%%%%%%%%%%%%%%%%%%%%%%%%
%PROOF
%%%%%%%%%%%%%%%%%%%%%%%%%%%%%%%%%%%%%%%%%%%%%%%%%%%%%%%%%
\begin{proof}
	Our first aim here is to prove 
	\begin{equation*}\label{Bresse Equation 28.8}
\int_{0}^{L}\eta|v^1_x|^2dx=o(1).
	\end{equation*} 
	For this sake, multiplying \eqref{Bresse Equation 6.8} by $\eta \overline{v_x^1}$	and integrating by parts, we get:
	\begin{align}\label{Bresse Equation 25.8}
&-i\int_0^L \lambda \rho_2 v^4 \eta^{\prime}\overline{v^1}dx - i \int_0^L \lambda \rho_2 v_x^4 \eta \overline{v^1}dx +\int_0^L (k_2 v^3_x+D_2v_x^4)(\eta\overline{v^1_{xx}})dx +\int_0^L (k_2 v_x^3+D_2 v_x^4)(\eta^{\prime}\overline{v_x^1})dx \nonumber\\
&+\int_0^L\eta k_1|v_x^1|^2dx +\int_0^L\eta k_1v^3\overline{v^1_x}dx +\int_0^L\ell k_1\eta v^5\overline{v_x^1}dx= \int_0^L\rho_2\frac{f^4}{\lambda^4}\eta \overline{v_x^1}dx.
	\end{align}
	%%%%%%%%%%%%%%%%%%%%%%%%%%
	Now, we need to estimate each term of \eqref{Bresse Equation 25.8}:\\
	$\bullet$ Using \eqref{Bresse Equation 9.8}, \eqref{Bresse Equation 19.8} and the fact that $f^3$ converges to zero in $H_0^1(0,L)$ (or $H_*^1(0,L)$), we get:  
	\begin{equation}\label{Bresse Equation 26.8}
-i\int_0^L \lambda \rho_2 v^4 \eta^{\prime}\overline{v^1}dx=-i\int_0^L\lambda\rho_2(i\lambda v^3-\dfrac{f^3}{\lambda^4})\eta^{\prime}\overline{v^1}=\int_0^L\rho_2\lambda^2v^3\eta^{\prime}\overline{v^1}+ i\int_0^L\dfrac{f^3}{\lambda^3}\eta^{\prime}\overline{v^1}=o(1).
	\end{equation}
	$\bullet$ Using \eqref{Bresse Equation 10.8} and the fact that $\lambda v^1$ is uniformly bounded in $L^2(0,L)$, we obtain:
	\begin{equation}
- i \int_0^L \lambda \rho_2 v_x^4 \eta \overline{v^1}dx=\frac{o(1)}{\lambda^2}.
	\end{equation}
	$\bullet$ From \eqref{Bresse Equation 4.8}, we remark that $\dfrac{1}{\lambda}v^1_{xx}$ is uniformly bounded in $L^2(0,L)$.
This fact combined with \eqref{Bresse Equation 10.8} and \eqref{Bresse Equation 12.8} yields
	\begin{equation}
\int_0^L (k_2\lambda v^3_x+D_2\lambda v_x^4)(\eta\dfrac{\overline{v^1_{xx}}}{\lambda})dx=\frac{o(1)}{\lambda}.
	\end{equation}
	$\bullet$ Using \eqref{Bresse Equation 10.8}, \eqref{Bresse Equation 12.8} and the fact that $v_x^1$ is uniformly bounded in $L^2(0,L)$, we get:
	\begin{equation} 
\int_0^L (k_2 v_x^3+D_2 v_x^4)(\eta^{\prime}\overline{v_x^1})dx=\frac{o(1)}{\lambda^2}
.
	\end{equation}
	$\bullet$ Using \eqref{Bresse Equation 9.8} and the fact that $v_x^1$ is uniformly bounded in $L^2(0,L)$, we obtain: 
	\begin{equation}
\int_0^L\eta k_1v^3\overline{v^1_x}dx +\int_0^L\ell k_1\eta v^5\overline{v_x^1}dx=o(1).
	\end{equation}
	$\bullet$ Using the fact that $f^4$ converges to zero in $L^2(0,L)$ and $v_x^1$ is uniformly bounded in $L^2(0,L)$, we get: 
	\begin{equation} \label{Bresse Equation 27.8}
\int_0^L\rho_2\frac{f^4}{\lambda^4}\eta \overline{v_x^1}dx=\frac{o(1)}{\lambda^4}.
	\end{equation}
	Finally, inserting equations \eqref{Bresse Equation 26.8}-\eqref{Bresse Equation 27.8} into \eqref{Bresse Equation 25.8} and  using the definition of $\eta$, 
	we get the desired estimates in \eqref{Bresse Equation 24a.8}.\newline
	Next, our second aim is to prove
	\begin{equation*}\label{Bresse Equation 29.8} 
\int_{0}^{L}\eta|\lambda v^1|^2dx=o(1).
	\end{equation*}
	For this, multiplying \eqref{Bresse Equation a.8} by   $\rho_1^{-1} \eta \overline{v^1}$ and integrating by parts, we get:
	\begin{align}\label{Bresse Equation 36.8}
\int_0^L \eta |\lambda v^1|^2dx =& \rho_1^{-1}\int_0^L k_1(v^1_x+v^3+\ell v^5)(\eta^{\prime}\overline{v^1}+\eta \overline{v_x^1})dx 
-\rho_1^{-1}\int_0^L \ell k_3 (v^5_x-\ell v^1)\eta \overline{v^1}\\
&-\int_0^L\left(\frac{f^2}{\lambda^4}+i\frac{f^1}{\lambda^3}\right)\eta \overline{v^1}dx.  \nonumber
	\end{align} 
	%%%%%%%%%%%%%%%%%%%%%
	So, using \eqref{Bresse Equation 9.8}, \eqref{Bresse Equation 24a.8}, the fact that $(v_x^1+v^3+\ell v^5)$, $(v_x^5-\ell v^1)$
are uniformly bounded in $L^2(0,L)$ and $f^1$, $f^2$ converge respectively to zero in $H_0^1(0,L)$, $L^2(0,L)$ in the right hand side of the above equation
and using the definition of $\eta$, we get the desird estimates in  \eqref{Bresse Equation 24b.8}.
\end{proof} 

%%%%%%%%%%%%%%%%%%%%%%%%%%%%%%%%%%
%LEMMA
%%%%%%%%%%%%%%%%%%%%%%%%%%%%%%%%
\begin{lem}
	Under all the above assumptions, we have:
	\begin{equation}\label{Bresse Equation 24.8} 
\int_0^L\eta |v_x^1|^2dx=\frac{o(1)}{\lambda^2}, 
\quad 
\int_{\alpha + \epsilon}^{\beta-\epsilon} |v_x^1|^2dx=\frac{o(1)}{\lambda^2}
	\end{equation}
	and
	\begin{equation}\label{Bresse Equation 24N.8}
\int_0^L\eta |\lambda v^1|^2dx=\frac{o(1)}{\lambda^2}
\quad 
\int_{\alpha + \epsilon}^{\beta-\epsilon}|\lambda v^1|^2dx=\frac{o(1)}{\lambda^2}.
	\end{equation}
\end{lem}
%%%%%%%%%%%%%%%%%%%%%%%%%%%%%%%%%%%
%PROOF
%%%%%%%%%%%%%%%%%%%%%%%%%%%%%%%%%%%
\begin{proof}
	For the clarity of the proof, we divide the proof into several steps:\\
	{\bf Step 1.} In this step, we will prove
	\begin{equation} \label{ Bresse Result1.8}
\rho_1 \int_0^L\eta|\lambda v^1|^2dx-k_1\int_0^L\eta|v_x^1|^2dx+ \ell (k_1+k_3)\mathrm{Re}\left\{ \int_0^L \eta v_x^5\overline{v_1}dx\right\}= \dfrac{o(1)}{\lambda^2}. 
	\end{equation}
	For this sake, multiplying \eqref{Bresse Equation a.8} by $\eta \overline{v^1}$ and integrating by parts, we get:
	\begin{align}\label{ Bresse equation Result1a.8}
&\rho_1 \int_0^L\eta|\lambda v^1|^2dx - k_1\int_0^L \eta |v_x^1|^2dx -k_1\mathrm{Re}\bigg\{\int_0^L \eta^{\prime} v_x^1\overline{v^1}dx \bigg\}+ k_1\mathrm{Re}\left\{\int_0^L\eta v_x^3\overline{v^1}dx\right\} \nonumber
\\&
+ \ell (k_1+k_3)\mathrm{Re}\left\{ \int_0^L \eta v_x^5\overline{v_1}dx\right\} -\ell^2 k_3\int_0^L\eta |v^1|^2dx=\frac{o(1)}{\lambda^4}.
	\end{align}
	Now, we need to estimate some terms of \eqref{ Bresse equation Result1a.8} as follows: \newline
	$\bullet$  We get after integrating by parts 
	\begin{equation*}
-k_1\mathrm{Re}\bigg\{\int_0^L \eta^{\prime} v_x^1\overline{v^1}dx \bigg\}=-\dfrac{k_1}{2}\int_0^L \eta^{\prime} (|v^1|^2)_xdx  = \dfrac{k_1}{2}\int_0^L \eta^{''} |v^1|^2dx.
	\end{equation*}
	Using \eqref{Bresse Equation 24b.8} in the previous equation, we obtain:
	\begin{equation} \label{ Bresse equation  Result1b.8}
-k_1\mathrm{Re}\bigg\{\int_0^L \eta^{\prime} v_x^1\overline{v^1}dx \bigg\}=\frac{o(1)}{\lambda^2}.
	\end{equation}
	$\bullet$ Using  \eqref{Bresse Equation 12.8} and  \eqref{Bresse Equation 24b.8}, we deduce that
	\begin{equation} \label{ Bresse equation  Result1c.8}
k_1\mathrm{Re}\left\{\int_0^L\eta v_x^3\overline{v^1}dx\right\} -\ell^2 k_3\int_0^L\eta |v^1|^2dx=\frac{o(1)}{\lambda^2}.
	\end{equation} 
	Finally, inserting \eqref{ Bresse equation  Result1b.8} and  \eqref{ Bresse equation  Result1c.8} in \eqref{ Bresse equation Result1a.8}, we get the desired estimate \eqref{ Bresse Result1.8}.
	\newline
	{\bf Step 2.} In this step, we will prove
	\begin{align}\label{Bresse Result 2.8}
&\left(\dfrac{k_1+k_3}{k_1}\right)\mathrm{Re}\bigg\{\int_0^L (k_2v_x^3+D_2v_x^4)\eta \overline{v^1_{xx}})dx\bigg\}+ (k_1+k_3)\int_0^L\eta |v_x^1|^2dx \nonumber
\\&
-\ell(k_1+k_3)\mathrm{Re}\bigg\{\int_0^L\eta v_x^5\overline{v^1}dx\bigg\}=\dfrac{o(1)}{\lambda^2}.
	\end{align}
	In order to prove \eqref{Bresse Result 2.8}, multiplying \eqref{Bresse Equation b.8} by the multiplier  $-\bigg(\dfrac{k_1+k_3}{k_1}\bigg)\eta \overline{v_x^1}$ and integrating by parts, we get: 
	\begin{align}\label{Bresse equation Result 2a.8}
&\underbrace{-\rho_2\bigg(\dfrac{k_1+k_3}{k_1}\bigg)\mathrm{Re}\bigg\{\int_0^L\eta\lambda^2 v^3 \overline{v_x^1}dx\bigg\}}_{I_1}\underbrace{-\bigg(\dfrac{k_1+k_3}{k_1}\bigg)\mathrm{Re}\bigg\{\int_0^L(k_2v_x^3+D_2v_x^4)_x\eta \overline{v_x^1}dx\bigg\}}_{I_2} \nonumber\\
& +(k_1+k_3)\int_0^L\eta|v_x^1|^2dx+ \underbrace{(k_1+k_3)\mathrm{Re}\bigg\{\int_0^L\eta v^3 \overline{v_x^1}dx\bigg\}}_{I_3}-\ell(k_1+k_3)\mathrm{Re}\bigg\{\int_0^L\eta v_x^5\overline {v^1}dx\bigg\}=\dfrac{o(1)}{\lambda^2}.
	\end{align}
	Next, we need to estimate $I_1, I_2$ and $I_3$. 
	\newline
	$\bullet$ Integrating  by parts $I_1$ and then using \eqref {Bresse Equation 12.8}, \eqref{Bresse Equation 19.8} and \eqref{Bresse Equation 24b.8}, we deduce that: \\
	\begin{equation}\label{Bresse equation Result 2b.8}
I_1= \rho_2\bigg(\dfrac{k_1+k_3}{k_1}\bigg)\mathrm{Re}\bigg\{\int_0^L\left(\eta^{\prime}\lambda v^3 \lambda\overline{v^1}+ \eta \lambda v_x^3 \lambda\overline v^1\right)dx\bigg\}=\dfrac{o(1)}{\lambda^2}.
	\end{equation}
	%Using \eqref {Bresse Equation 12.8}, \eqref{Bresse Equation 19.8} and \eqref{Bresse Equation 24a.8}, we deduce that 
	%\begin{equation}
	%I_1= \dfrac{o(1)}{\lambda^2}.
	%\end{equation}
	%
	$\bullet$
	Integrating  by parts $I_2$ and then using \eqref{Bresse Equation 10.8}, \eqref {Bresse Equation 12.8} and \eqref{Bresse Equation 24a.8}, we get: \\
	\begin{align}\label{Bresse equation Result 2c.8}
I_2&=\bigg(\dfrac{k_1+k_3}{k_1}\bigg)\mathrm{Re}\bigg\{\int_0^L\left(k_2v_x^3+D_2v_x^4\right)\eta^{\prime} \overline{v_x^1}dx\bigg\}dx+\bigg(\dfrac{k_1+k_3}{k_1}\bigg)\mathrm{Re}\bigg\{\int_0^L(k_2v_x^3+D_2v_x^4)\eta \overline{v_{xx}^1}dx\bigg\}dx\nonumber\\
&=\bigg(\dfrac{k_1+k_3}{k_1}\bigg)\mathrm{Re}\bigg\{\int_0^L(k_2v_x^3+D_2v_x^4)\eta \overline{v_{xx}^1}dx\bigg\}dx+\dfrac{o(1)}{\lambda^2}. 
	\end{align}
	$\bullet$
	By using  \eqref{Bresse Equation 19.8} and \eqref{Bresse Equation 24a.8}, we deduce that 
	\begin{equation}\label{Bresse equation Result 2d.8}
I_3=(k_1+k_3)\mathrm{Re}\bigg\{\int_0^L\eta v^3 \overline{v_x^1}dx\bigg\}=\dfrac{o(1)}{\lambda^2}.
	\end{equation}
	Finally, inserting \eqref{Bresse equation Result 2b.8}, \eqref{Bresse equation Result 2c.8}, and \eqref{Bresse equation Result 2d.8} into \eqref{Bresse equation Result 2a.8},
we get the desired estimate \eqref{Bresse Result 2.8}.

	\noindent{\bf Step3.} Combining \eqref{ Bresse Result1.8} and \eqref{Bresse Result 2.8}, we get
	\begin{equation}\label{Bresse equation Result3.8}
\rho_1\int_0^L\eta|\lambda v^1|^2dx + k_3 \int_0^L\eta |v_x^1|^2dx+\bigg(\dfrac{k_1+k_3}{k_1}\bigg)\mathrm{Re}\bigg\{\int_0^L\left(k_2v_x^3+D_2v_x^4\right)\eta \overline{v_{xx}^1}dx\bigg\}= \dfrac{o(1)}{\lambda^2}.
	\end{equation}
	{\bf Step 4.} In this step, we conclude the proof of the main estimates \eqref{Bresse Equation 24.8} and \eqref{Bresse Equation 24N.8}. For this aim, multiplying
\eqref{Bresse Equation a.8} by $\eta \left(k_2 \overline{ v_x^3} +D_2 \overline{v_x^4}\right) $, we get:
	\begin{align}\label{Bresse equation Result3a.8}
&\mathrm{Re}\left\{\int _0^L \rho_1\eta \lambda v^1\lambda\left(k_2 \overline{ v_x^3} +D_2 \overline{v_x^4}\right)dx\right\} +k_1\mathrm{Re}\left\{ \int_0^L\eta \left(k_2 \overline{ v_x^3} +D_2 \overline{v_x^4}\right) v_{xx}^1dx\right\} \nonumber\\
&\underbrace{+k_1\mathrm{Re}\left\{ \int_0^L\eta\left(v_x^3+\ell v_x^5\right) \left(k_2 \overline{ v_x^3} +D_2 \overline{v_x^4}\right) dx\right\} +\ell k_3\mathrm{Re}\left\{ \int_0^L\left(v_x^5-\ell v^1\right)\eta \left(k_2 \overline{ v_x^3} +D_2 \overline{v_x^4}\right) dx\right\}}_{I}=\frac{o(1)}{\lambda^2}.
	\end{align}
	Using the fact that $v_x^3$ and $(v_x^5-\ell v^1)$ are uniformly bounded in $L^2(0,L)$, \eqref{Bresse Equation 10.8} and \eqref{Bresse Equation 12.8}, we get:
	\begin{align}\label{Bresse equation Result3b.8}
I&=k_1\mathrm{Re}\left\{ \int_0^L\eta\left(v_x^3+\ell v_x^5\right) \left(k_2 \overline{ v_x^3} +D_2 \overline{v_x^4}\right) dx\right\} +\ell k_3\mathrm{Re}\left\{ \int_0^L\left(v_x^5-\ell v^1\right)\eta \left(k_2 \overline{ v_x^3} +D_2 \overline{v_x^4}\right) dx\right\}\nonumber\\
&=\frac{o(1)}{\lambda^2}.
	\end{align}
	Substitute \eqref{Bresse equation Result3b.8} in \eqref{Bresse equation Result3a.8}, we get:
	\begin{equation}\label{Bresse equation Result3c.8}
k_1\mathrm{Re}\left\{ \int_0^L\eta \left(k_2 \overline{ v_x^3} +D_2 \overline{v_x^4}\right) v_{xx}^1dx\right\}=-\mathrm{Re}\left\{\int _0^L \rho_1\eta \lambda v^1\lambda\left(k_2 \overline{ v_x^3} +D_2 \overline{v_x^4}\right)dx\right\}+\frac{o(1)}{\lambda^2}.
	\end{equation}
	Now, substitute \eqref{Bresse equation Result3c.8} in \eqref{Bresse equation Result3.8}, we obtain:
	\begin{equation}\label{Bresse equation Result3d.8}
\rho_1\int_0^L\eta|\lambda v^1|^2dx + k_3 \int_0^L\eta |v_x^1|^2dx=-\bigg(\dfrac{k_1+k_3}{k_1^2}\bigg)\mathrm{Re}\left\{\int _0^L \rho_1\eta \lambda v^1\lambda\left(k_2 \overline{ v_x^3} +D_2 \overline{v_x^4}\right)dx\right\}+\frac{o(1)}{\lambda^2}.
	\end{equation}
	%%%%%%%%%%%%
	% NEW SG
	%%%%%%%%%%%%
	We will now apply Young's inequality in \eqref{Bresse equation Result3d.8}. For this sake , let $\epsilon>0$ be given. We get:
	%%%%%%%
	% END SG
	%%%%%%%
	\begin{align*}\label{Bresse equation Result3e.8}
\rho_1\int_0^L\eta|\lambda v^1|^2dx + k_3 \int_0^L\eta |v_x^1|^2dx
&\leq \underbrace{\dfrac{1}{\epsilon}\bigg(\dfrac{k_1+k_3}{k_1^2}\bigg)^2\int _0^L \rho_1\eta \lambda^2\left|  k_2 \overline{ v_x^3} +D_2 \overline{v_x^4}\right|^2dx}_{=\frac{o(1)}{\lambda^2}} + \epsilon \int_0^L \rho_1 \eta |\lambda v^1|^2dx+\frac{o(1)}{\lambda^2}
\\
& \leq \epsilon \int_0^L \rho_1 \eta |\lambda v^1|^2dx+\frac{o(1)}{\lambda^2}.
	\end{align*}
	Consequently, we have:
	\begin{equation*}
(1-\epsilon)\rho_1\int_0^L\eta|\lambda v^1|^2dx + k_3 \int_0^L\eta |v_x^1|^2dx=\frac{o(1)}{\lambda^2}.
	\end{equation*}
	Finally, it is sufficient to take $\epsilon =\dfrac{1}{2}$ in the previous equation to  get the desired estimates in \eqref{Bresse Equation 24.8}
	and \eqref{Bresse Equation 24N.8}. The proof is thus complete.
\end{proof}
%%%%%%%%%%%%%%%%%%%%%%%%%%%%%%%%%%%%%%%%%%%%%%%%%%%%%%%%%%%%%%%%%%%%%%%%%%%%%%%%%%%%%%%%%%%%%%%%%%%%%%%%%%%%%%%%%
% LEMMA
%%%%%%%%%%%%%%%%%%%%%%%%%%%%%%%%%%%%%%%%%%%%%%%%%%%%%%%%%%%%%%%%%%%%%%%%%%%%%%%%%%%%%%%%%%%%%%%%%%%%%%%%%%%%%%%%%%
\begin{lem} \label{lemma Bresse polynomial4.8}
	Under all the above assumptions, we have:
	\begin{equation}\label{Bresse Equation 52.8} 
\int_{\alpha+\epsilon}^{\beta-\epsilon} |v^5_x|^2dx=o(1) 
	\end{equation}
\end{lem}
%%%%%%%%%%%%%%%%%%%%%%%%%%%%%%%%%%%%%%%%%%%%%%%%%%%%%%%%%
\begin{proof}
	%First, substituting \eqref{Bresse Equation 3.8} into \eqref{Bresse Equation 4.8}, we get:
	%\begin{equation}\label{Bresse Equation 53.8}
	%-\lambda^2 \rho_1 v^1- k_1(v^1_x+v^3+\ell v^5)_x-\ell k_3(v^5_x-\ell v^1)=\rho_1\left(\frac{f^2}{\lambda^4}+i\frac{f^1}{\lambda^3}\right).
	%\end{equation}
	%%%%%%%%%%%%%%%%%%
	Multiplying \eqref{Bresse Equation a.8} by $\eta \overline {v_x^5}$ and integrating over $(0,L)$, we get:
	\begin{align}
(\ell  k_1+ \ell  k_3) \int_0^L \eta |v^5_x|^2 =&-\rho_1\int_0^L\eta \lambda^2v^1\overline{v_x^5}dx+k_1 \int_0^L v_x^1 \eta^{\prime}\overline{v^5_x}dx+ k_1 \int_0^L \eta \lambda v_x^1 \frac{\overline{v_{xx}^5}}{\lambda}dx\\
&-k_1 \int_0^L \eta v_x^3 \overline{v_x^5}dx +\ell ^2k_3 \int_0^L\eta v^1 \overline{v_x^5}dx+ \frac{o(1)}{\lambda^3}.\nonumber
	\end{align}
	Finally, using \eqref{Bresse Equation 12.8}, \eqref{Bresse Equation 24.8}, \eqref{Bresse Equation 24N.8}, the fact that $v_x^5$ is uniformly bounded in $L^2(0,L)$
and $\dfrac{1}{\lambda}v^5_{xx}$ is uniformly bounded in $L^2(0,L)$ due to \eqref{Bresse Equation 8.8} in the right hand side of the previous equation,
we get the desired  {estimate} \eqref{Bresse Equation 52.8}. 
	The proof is thus complete. 
\end{proof}
%%%%%%%%%%%%%%%%%%%%%%%%%%%%%%%%%%%%%%%%%%%%%%%%%%%%%%%%%%%%%%%%%%%%%%%%%%%%%%%%%%%%%%%%%%%%%%%%%%%%%%%%%%%%%%%%%%
%LEMMA
%%%%%%%%%%%%%%%%%%%%%%%%%%%%%%%%%%%%%%%%%%%%%%%%%%%%%%%%%%%%%%%%%%%%%%%%%%%%%%%%%%%%%%%%%%%%%%%%%%%%%%%%%%%%%%%%%%
\begin{lem} \label{lemma Bresse polynomial5.8}
	Under all the above assumptions, we have:
	\begin{equation}\label{Bresse Equation 54.8} 
\int_{\alpha+\epsilon}^{\beta-\epsilon} |\lambda v^5|^2dx=o\left(1\right).
	\end{equation}
\end{lem}
%%%%%%%%%%%%%%%%%%%%%%%%%%%%%%%%%%%%%%%%%%%%%%%%%%%%%%%%%
\begin{proof}
	Multiplying \eqref{Bresse Equation c.8} by   $ \eta \rho_1^{-1}\overline{v^5}$, we get:
	\begin{align}\label{Bresse Equation 57.8}
\int_0^L \eta |\lambda v^5|^2dx =& \rho_1^{-1}\int_0^Lk_3(v^5_x-\ell v^1) (\eta^{\prime}\overline{v^5}+\eta \overline{v_x^5})dx 
+\rho_1^{-1}\int_0^L \ell k_1(v^1_x+v^3+\ell v^5) \eta \overline{v^1}dx\\
&-\int_0^L\left(\frac{f^6}{\lambda^4}+i\frac{f^5}{\lambda^3}\right)\eta \overline{v^5}dx.  \nonumber
	\end{align} 
	%%%%%%%%%%%%%%%%%%%%%
	Using \eqref{Bresse Equation 9.8}, \eqref{Bresse Equation 52.8}, the fact that $(v^1_x+v^3+\ell v^5)$, $(v^5_x-\ell v^1)$ are uniformly bounded in $L^2(0,L)$,  $f^5$, $f^6$ converge to zero  respectively in $H_0^1(0,L)$ (or in $H_*^1(0,L)$), $L^2(0,L)$ in the right hand side of the above equation, we deduce:
	\begin{equation*}
\int_0^L \eta |\lambda v^5|^2dx =o(1).
	\end{equation*}
	Finally, using the definition of $\eta$, we get the desired  {estimate} \eqref{Bresse Equation 54.8}. The proof is thus complete. 
\end{proof}
\textbf{Proof of Theorem \ref{PolyThe1.2}} 
 It follows from Lemmas \ref{Bresse polydiss2.1}, \ref{Lemma  Bresse Poly.8},  \ref{lemma Bresse polynomial1.8}, \ref{lemma Bresse polynomial4.8} and  \ref{lemma Bresse polynomial5.8} that $\norm{U_n}_{\mathcal{\HH}_j}=o(1)$ on $(\alpha +\epsilon, \beta-\epsilon).$
 So one can use estimate \eqref{Bresse equation 2.26} with $D_1=D_3=0$ and Lemma \ref{lemma6.7}
 to conclude that $\norm{U_n}_{\mathcal{\HH}_j}=o(1)$ on $(0,L)$ which is a contradiction
 with \eqref{Bresse Equation 1.8}. Consequently, condition (H4) holds and the energy of smooth solutions of system \eqref{eqq1.1'}
 decays polynomially as $t$ goes to infinity.  
\hfill{$\square$}

%%%%%%%%%%%%%%%%%%%%%%%%%%%%%%%%%%%%%%%%%%%%%%%%%%
%%%%%%%%%%%%%%%%%%%%%%%%%%%%%%%%%%%%%%%%%%%%%%%%%%
%%%%%%%%%%%%%%%%%%%%%%%%%%%%%%%%%%%%%%%%%%%%%%%%%%
% SECTION: Lack of exponential %
%%%%%%%%%%%%%%%%%%%%%%%%%%%%%%%%%%%%%%%%%%%%%%%%%%
%%%%%%%%%%%%%%%%%%%%%%%%%%%%%%%%%%%%%%%%%%%%%%%%%%	
%%%%%%%%%%%%%%%%%%%%%%%%%%%%%%%%%%%%%%%%%%%%%%%%%%
\section{Lack of exponential  stability }\label{Bresse Section 2.4}\label{NonUni}
It was proved that the Bresse system subject to one or two viscous dampings is exponentially stable if and only if the wave propagate at the same speed
(see \cite{wehbey} and \cite{Ghader}). In the case of viscoelastic damping, the situation is more delicate. In this section, we prove that the
Bresse system \eqref{eqq1.1'}-\eqref{DNND} subject to two global viscoelastic dampings is not exponentially stable even if the waves propagate at same speed.
So, we assume that:
\begin{equation}\label{damping1}
	D_1= 0 \quad \mathrm{and} \quad D_2=D_3=1 \quad \mathrm{in} \quad (0,L). 
\end{equation}
%%%%%%%%%%%%%%%%%%%%%%%%%%%%%%%%%%%%%%%%%%%%%%%%%%
\begin{thm}\label{theorem lackexp1}
	Under hypothesis \eqref{damping1}, the Bresse system \eqref{eqq1.1'}-\eqref{DNND}, is not exponentially stable in the energy space $\HH_2$.
\end{thm}
%%%%%%%%%%%%%%%%%%%%%%%%%%%%%%%%%%%%%%%%%%%%%%%%%%%
%PROOF OF LACK OF EXPONENTIAL
%%%%%%%%%%%%%%%%%%%%%%%%%%%%%%%%%%%%%%%%%%%%%%%%%%%%
\begin{proof}
	For the proof of Theorem \ref{theorem lackexp1},  it suffices to show that there exists 
	\begin{itemize}
\item a sequence  $(\lambda_n)\subset \rr$ with $\displaystyle \lim_{n\to+\infty}\abs{\lambda_n}=+\infty$, and
\item a sequence  $(V_n)\subset D(\AA_2)$, 
	\end{itemize}
	such that $(i\lambda_n I-\AA_2)V_n$ is bounded in $\HH_2$ and $\displaystyle\lim_{n\to+\infty}\norm{V_n}=+\infty$. For the sake of clarity, we skip the index $n$.
	Let $F=(0,0,0,f_4,0,0)\in \HH_2$
	with
	$$ f_4(x)=\cos\left( \frac{n\pi x}{L}\right), \ \ 
	\lambda = \frac{n\pi \sqrt{\rho_2 k_2}}{L \rho_2},\ \, n\in \nn.$$
	%%%%%%%%%%%%%%%%%%%%%%%%%%%%%%%%%%%%%%%%%%%%%%%%
	We solve the following equations:
	\begin{equation}
i\lambda v^1-v^2= 0,\label{es1}
	\end{equation}
	\begin{equation}
i \lambda \rho_1v^2-k_1\left(v_{xx}^1+v^3_x+\ell v^5_x\right) -\ell k_3\left(v_x^5-\ell v^1\right)-\ell \left(v_x^6-\ell v^2\right)=0 ,\label{es2}
	\end{equation}
	\begin{equation}
i\lambda v^3-v^4=0,\label{es3}
	\end{equation}
	\begin{equation}
i\lambda \rho_2 v^4 -k_2v_{xx}^3+k_1\left(v^1_x+v^3+\ell  v^5\right)=
\rho_2f_4 ,\label{es4}
	\end{equation}
	\begin{equation}
i \lambda v^5 -v^6=0,\label{es5}
	\end{equation}
	\begin{equation}
i\lambda \rho_1 v^6-k_3\left(v^5_{xx}-\ell  v^1_x\right) +\ell v^2_x+\ell  k_1\left(v^1_x+v^3+\ell  v^5\right) = 0.\label{es6}
	\end{equation}
	Eliminating $v^2$, $v^4$ and $v^6$ in (\ref{es2}), (\ref{es4}) and (\ref{es6}) by (\ref{es1}), (\ref{es3}) and (\ref{es5}),
	we get:
	\begin{equation}
\lambda^2 \rho_1v^1+k_1\left(v_{xx}^1+v^3_x+\ell v^5_x\right)  +\ell \left(k_3+i \lambda\right)\left(v_x^5-\ell v^1\right) =0,
\label{es7}
	\end{equation}
	\begin{equation}
\lambda^2 \rho_2v^3+k_2 v^3_{xx}  - k_1\left( v^1_x+v^3+\ell v^5\right)= -\rho_2 f^4,\label{es8}
	\end{equation}
	\begin{equation}
\lambda^2 \rho_1 v^5+k_3 \left( v^5_{xx}-\ell v^1_x\right )-i \lambda \ell v^1_{x} -\ell  k_1\left(v^1_x+v^3+\ell  v^5\right) = 0 .\label{es9}
	\end{equation}
	This can be solved by the ansatz:
	\begin{equation}
v^1=A\sin\left(\frac{n\pi x}{L}\right), \ \ \
v^3=B\cos\left(\frac{n\pi x}{L}\right), \ \ \ 
v^5=C\cos\left(\frac{n\pi x}{L}\right) \label{ess}
	\end{equation}
	where $A$, $B$ and $C$ depend on $\lambda$ are constants to be determined. Notice that $k_2\left(\frac{n \pi}{L} \right)^2 -\rho_2 \lambda^2=0$,
and inserting (\ref{ess}) in (\ref{es7})-(\ref{es9}) we obtain that:
	\begin{equation}
\left( \left(\frac{n \pi}{L}\right)^2k_1-\lambda^2 \rho_1+\left(k_3+i \lambda\right)\ell ^2\right)A + k_1 \left(\frac{n \pi}{L}\right)B + \left( k_1+k_3+i\lambda \right)\ell  \left(\frac{n \pi}{L}\right) C= 0 ,\label{es10}
	\end{equation}
	\begin{equation}
k_1 \left(\frac{n \pi}{L}\right) A + k_1 B +\ell k_1C= \rho_2,\label{es11}
	\end{equation}
	\begin{equation}
\left(k_1+k_3+i\lambda\right) \ell  \left(\frac{n \pi}{L}\right) A + \ell  k_1 B + \left[k_3\left(\frac{n \pi}{L}\right)^2-\lambda^2\rho_1 +\ell ^2k_1\right]C= 0.\label{es12}
	\end{equation}
	Equivalently,
	\begin{equation}
\begin{pmatrix}
\left(\frac{n \pi}{L}\right)^2k_1-\lambda^2 \rho_1+\left(k_3+i \lambda\right)\ell ^2& k_1 \left(\frac{n \pi}{L}\right) &  \left( k_1+k_3+i\lambda \right)\ell  \left(\frac{n \pi}{L}\right) \\
k_1 \left(\frac{n \pi}{L}\right) &  k_1 &\ell k_1\\
\left(k_1+k_3+i\lambda\right) \ell  \left(\frac{n \pi}{L}\right) & \ell  k_1 &k_3\left(\frac{n \pi}{L}\right)^2-\lambda^2\rho_1 +\ell ^2k_1
\end{pmatrix}
\begin{pmatrix}
A\\
B\\
C
\end{pmatrix}
=
\begin{pmatrix}
0\\
\rho_2 \\
0
\end{pmatrix}.
	\end{equation}
	This implies that:
	\begin{equation}\label{ess1}
A= \dfrac{\left(k_2 \rho_1 - \rho_2 k_3\right)\rho_2^2L}{\pi \left(k_2 \rho_1^2 - k_3 \rho_1 \rho_2 + \rho_2 \ell ^2\right)k_2 n}+O(n^{-2}),
	\end{equation}
	\begin{equation}\label{ess2}
B=\dfrac{\rho_2 \left(k_1 k_3 \rho_2^2 + \left(\left(-k_1-k_3\right)\rho_1+ \ell ^2\right)k_2 \rho_2 + k_2^2 \rho_1^2\right)}{k_1 \left(\left(-k_3 \rho_1 + \ell ^2\right)\rho_2 + k_2 \rho_1^2\right)k_2} + O(n^{-1}),
	\end{equation}
	\begin{equation}\label{ess3}
C= \dfrac{i \ell \rho_2^2 L \sqrt{\rho_2 k_2}}{\pi\left(\left(-k_3 \rho_1 +\ell ^2\right)\rho_2 + k_2 \rho_1^2\right)k_2 n}+ O(n^{-2}).
	\end{equation}
	Now, let $V_n=\left(v^1, \ i \lambda v^1,v^3, \ i \lambda v^3,v^5, \ i \lambda v^5 \right)$, where $v^1, v^3$ and $v^5$ are given by \eqref{ess} and
\eqref{ess1}-\eqref{ess3}. It is easy to check that
	$$\norm{V_n}_{\HH_2} \geq \sqrt{\rho_2}\norm{\lambda v^3} \sim |B \lambda| \sim |n| \rightarrow +\infty \ \ \mathrm{as} \ \ n \rightarrow +\infty.$$
	On the other hand, using \eqref{es1}-\eqref{es6}, we deduce that 
	$$\norm{(i\lambda I-\AA_2)V_n}_{\HH_2}^2= \norm{(0, 0, 0, \rho^2 f^4-i \lambda D_2 v^3_{xx}, 0, i \lambda D_3 v^5_{xx})}_{\HH_2}^2\leq c .$$
	Consequently, $\norm{(i\lambda I-\AA_2)V_n}_{\HH_2}^2$ is bounded as $n$ tends to $+\infty$. Thus the proof is complete.
\end{proof}
%%%%%%%%%%%%%%%%%%%%%%%%%%%%%%%%%%%%%%%%%%%%%
\begin{rk}
	By a similar way, we can prove that the Bresse system \eqref{eqq1.1'}-\eqref{DNND} subject to only one viscoelastic damping is also not exponentially
stable even if the waves propagate at same speed.  
	\hfill$\square$
\end{rk}
\section{Additional results and summary}\label{additional}
\paragraph{\textbf{Global Kelvin--Voigt damping : analytic stability}}
 In \cite{F.Huang1988}, Huang  considered a one-dimensional wave equation with global Kelvin-Voigt damping and he proved 
that the semigroup associated to the equation is not only exponentially stable, but also is analytic. So, it is logic that in the case of three 
waves equations with three global dampings, the decay will be also analytic.

In this part, we state the analytic stability of the Bresse systems \eqref{eqq1.1'}-\eqref{DDDD} and  \eqref{eqq1.1'}-\eqref{DNND}
provided that there exists a positive constant $d_0$ such that:
\begin{equation}\label{condition1}
	D_1, \,  D_2 \, \, D_3 \geq d_0 >0 \ \ {\rm{for \, every}} \ x \in (0,L).
\end{equation}
%%%%%%%%%%%%%%%%%%%%%%%%%%%%%%%%%%%%%%%%%%%%%%%%%%
% Theorem analytic stability%
%%%%%%%%%%%%%%%%%%%%%%%%%%%%%%%%%%%%%%%%%%%%%%%%%%
\begin{thm}\label{analytic}
	Assume that condition \eqref{condition1} holds. Then, the $C_0$-semigroup $e^{t\mathcal{A}_j}$, for $j=1,2,$ is analytically stable.
\end{thm}
%%%%%%%%%%%%%%%%%%%%%%%%%%%%%%%%%%%%%%%%%%%%%%%%%%%%%%%%%%%%%%%%%%
The proof relies on the characterization of  the analytic stability stated in theorem \ref{type} and on the same kind of proof  used for  the preceding results : 
we use a contradiction argument and much simpler estimation to obtain the result. This much simpler proof is left to the reader.
\paragraph{\textbf{Localized smooth damping : exponential stability}}
In \cite{LiuLiu-1998}, K. Liu and Z. Liu  considered a one-dimensional wave equation with Kelvin-Voigt damping distributed locally on any subinterval
of the region occupied by the beam. They proved that the semigroup associated with the equation for the transversal motion of the beam
is exponentially stable, although the semigroup associated with the equation for the longitudinal motion of the beam is not exponentially stable.

And in \cite{Liu-Liu-2002}, K. Liu and Z. Liu  reconsidered the one-dimensional linear wave equation with the Kelvin-Voigt damping presented on a subinterval but with
smooth transition at the end of the interval. They proved that the smoothness of the damping coefficient at the interface leads to an exponential stability. 
They were the first researchers  to suggest that discontinuity of material properties at the interface and the “type” of  the damping can affect the qualitative behavior of the energy decay.
The smoothness of the coefficient at the interface plays a crucial role in the stabilization of  the wave equation. In this part, we generalize these results on Bresse system.

So we consider the Bresse systems \eqref{eqq1.1'}-\eqref{DDDD} and \eqref{eqq1.1'}-\eqref{DNND} subject to three local viscoelastic Kelvin-Voigt dampings
with smooth coefficients at the interface. We establish uniform (exponential) stability of the $C_0$-semigroup $e^{t\AA_j}$, $j=1,2$.
For this purpose, let $\emptyset \not=\omega=(\alpha,\beta)\subset (0,L)$ be the biggest nonempty open subset of $(0,L)$ satisfiying:
\begin{equation} \label{3local}
	\exists\, d_0 >0 \,\,\, \mbox{such that} \,\, D_i \geq d_{0},\ \ {\rm{for \, almost \, \ every}} \ \ x \in \omega, \, \, i=1,2,3.
\end{equation}
\begin{thm}\label{Bresse Theorem 2.10}
	Assume that condition \eqref{3local} holds. Assume also that $D_1$, $D_2$,  $D_3\in W^{1,\infty}(0,L)$. Then, the $C_0$-semigroup $e^{t\mathcal{A}_j}$
is exponentially stable in $\HH_j$, $j=1,2$, {\it i.e.}, for all $U_0\in \HH_j$, there exist constants $M\geq1$ and $\delta >0$ independent of $U_0$ such that:
	\begin{equation*}
\left\|e^{t\mathcal{A}_j}U_0\right\|_{\mathcal{H}_j}\leq M e^{-\delta t}\left\|U_0\right\|_{\mathcal{H}_j}, \quad t\geq0,\,\, j=1,2.
	\end{equation*}
\end{thm}
Again, the proof  relies on the characterization of  the exponential stability stated in theorem \ref{type} and on the same kind of  arguments used for
the proof of the preceding results : we use a contradiction argument and simpler estimation to obtain the result. This proof is left to the reader.

\medskip

The following table summarizes the results of this study:

\begin{center}
	\begin{tabular}{|c|c|c|c|c|}
\hline
&&&&\\
Regularity of $D_1$ & Regularity of $D_2$ & Regularity of $D_3$ &  Localization  
&Energy decay rate \\
&&&&\\
\hline
&&&&\\
$ L^\infty{(0,L)}$&  $ L^\infty{(0,L)}$&  $ L^\infty{(0,L)}$& $\begin{array}{c}D_i\geq d_0>0 \ \ \mathrm{in} \ (0,L)\\ i=1,2,3 \end{array} $ & Analytic stability\\
&&&&\\
\hline
&&&&\\
$ W^{1,\infty}(0,L)$&$ W^{1,\infty}(0,L)$&$ W^{1,\infty}(0,L)$&  $\begin{array}{c}D_i\geq d_0>0 \ \ \mathrm{in} \ \omega\\ i=1,2,3 \end{array} $ & Exponential stability\\
&&&&\\
\hline 
$ L^\infty{(0,L)}$&  $ L^\infty{(0,L)}$&  $ L^\infty{(0,L)}$& $\displaystyle{\bigcap_{i=1}^{3}}$ \, supp$D_i=\overline{\omega}$ & Polynomial of type $\dfrac{1}{t}$\\
\hline
0&  $ L^\infty{(0,L)}$&  0&$ \ D_2\geq d_0>0 \ \ \mathrm{in} \ \omega$ & Polynomial of type $\dfrac{1}{\sqrt{t}}$\\
	 \hline
	\end{tabular}
\end{center}
\section*{Acknowledgments}
The authors thanks professor Kais Ammari for his valuable discussions and comments.\\
Chiraz Kassem would like to thank the AUF agency for its support in the framework of the PCSI project untitled {\it Theoretical and Numerical Study of Some Mathematical Problems and Applications}.\\
Ali Wehbe would like to thank the CNRS and the LAMA laboratory of Mathematics of the Universit\'e Savoie Mont Blanc for their supports. 

The authors thank also the referees for very useful comments.  
\bibliographystyle{siam}
%\bibliography{Bresse.bib}

\end{document}